\documentclass[12pt, a4paper]{article}

\usepackage{a4,amsmath,amssymb, amsthm, latexsym, color, graphicx,url}
\usepackage{amsfonts}
\usepackage{tikz}
\usepackage{float}
\usepackage{psfrag}
\usepackage{verbatim}
\usetikzlibrary{arrows.meta}
\usetikzlibrary{decorations.markings}
\usetikzlibrary{calc}

\usepackage{authblk}
\usepackage{fullpage}
\usepackage{enumitem}
\usepackage{xcolor}
\newtheorem{theorem}{Theorem}[section]
\newtheorem{proposition}[theorem]{Proposition}
\newtheorem{lemma}[theorem]{Lemma}
\newtheorem{corollary}[theorem]{Corollary}

\theoremstyle{definition}
\newtheorem{example}[theorem]{Example}
\newtheorem{definition}[theorem]{Definition}
\newtheorem{remark}[theorem]{Remark}

\newtheorem{conjecture}[theorem]{Conjecture}

\newtheorem*{TP}{Tiling Problem}

\title{On optimal $(v,5,3,1)$ optical orthogonal codes}
\author[1]{Laura Johnson}
\author[2]{Lorenzo Mella}
\author[3]{Anita Pasotti}
\author[4]{Emma Smith}
\affil[1]{School of Mathematics, University of Bristol, Bristol, BS8 1UG, UK, \texttt{laura.marie.johnson@bristol.ac.uk}}
\affil[2]{Dip. di Scienze Fisiche, Informatiche, Matematiche, Universit\`a degli Studi di Modena e Reggio Emilia, Via Campi 213/A, I-41125 Modena, Italy, \texttt{lorenzo.mella@unipr.it}}
\affil[3]{DICATAM - Sez. Matematica, Universit\`a degli Studi di Brescia, Via Branze 43, I-25123 Brescia, Italy, \texttt{anita.pasotti@unibs.it}}
\affil[4]{Information Security Group, Royal Holloway, University of London, Egham, Surrey, TW20 0EX, UK, \texttt{Emma.Smith.2020@live.rhul.ac.uk}}
\begin{document}
\maketitle

\begin{abstract}
In this paper, we establish a tight upper bound for the size of a $(v,5,3,1)$ optical orthogonal code (OOC) and call such a code optimal when it achieves this upper bound. In addition, if an optimal OOC has the property that all non-identity elements of $\mathbb{Z}_v$ are covered by the list of differences of its codeword-sets then it is called perfect. In this paper, we present direct and recursive constructions for infinite classes of  optimal and perfect $(v,5,3,1)$-OOCs. 
\end{abstract}

\section{Introduction}
Optical orthogonal codes (OOCs, for short) were originally investigated in the context of multi-access communication systems \cite{BW,CSW}. 
Although OOCs are usually defined as a collection of binary sequences, for the purpose of this paper we prefer to adopt an equivalent set notation. A $(v,k,\lambda_a,\lambda_c)$-OOC is a collection $\mathcal{C} = \{C_1,\ldots,C_s\}$ of $k$-subsets of $\mathbb{Z}_v$ (called \textit{codeword-sets}) satisfying the following conditions: 
\begin{gather}
   |C_i \cap (C_i+t)| \leq \lambda_a \,\ \ {\rm for}\,\ 1 \leq i \leq s \,\ \ {\rm and}\,\ \ 1 \leq t \leq v-1; \label{AutoProp}\\
  |C_i \cap (C_j+t)| \leq \lambda_c \,\ \ {\rm for}\,\  1 \leq i < j \leq s \,\ \ {\rm and}\,\ \ 0 \leq t \leq v-1. \label{CrossProp}
\end{gather}
Conditions (\ref{AutoProp}) and (\ref{CrossProp}) are referred to as the \emph{autocorrelation property} and the \emph{cross-correlation property}, respectively. The size $k$ of the subsets is also called the \emph{weight} of the codeword-sets. The integer $s$, namely the number of codeword-sets in $\mathcal{C}$, is called the \emph{size} of the code.
A $(v,k,\lambda_a,\lambda_c)$-OOC with $\lambda_a=\lambda_c=\lambda$ is simply denoted as a $(v,k,\lambda)$-OOC.
The most investigated class of OOCs is that of $(v,k,1)$-OOCs with large size. This interest is motivated both by their applications and by their strong connections with difference families, see for example \cite{AB1,B,BP2,CZ,CZ2,W}. Several results have also been obtained for $(v,k,2)$-OOCs; see, for instance, \cite{AM,CC,FCJ}.

Despite the considerable attention devoted to  $(v,k,\lambda)$-OOCs, the case in which the two correlation parameters are distinct is equally worthy of investigation.  From both an applied and a theoretical viewpoint, the most significant and challenging situation arises when the two correlation parameters are small. Indeed, low autocorrelation enables correct decoding in an asynchronous communication system,  whilst low cross-correlation prevents collisions between users; furthermore, maximising the number of codeword-sets means that more users can use a given bandwidth. For these reasons, many papers have focused on $(v,k,\lambda_a,1)$-OOCs of large size. In \cite{YF}, it is shown  that the size of a $(v,k,\lambda_a,1)$-OOC, with $\lambda_a\neq 1$, cannot exceed 
$\lambda_a \lfloor \frac{v-1}{k} \rfloor$ but this bound is, in general, far from being tight. Moreover, we recall that a $(v,k,k-1,1)$-OOC is nothing but a conflict-avoiding code, introduced in \cite{TR} and subsequently investigated in many papers. It was therefore natural to focus on  $(v,4,2,1)$-OOCs, see \cite{BT12,BMP,MB,WC}, and $(v,5,2,1)$-OOCs, see \cite{BT13,BPW}. In particular, a sharp upper bound for the size of a $(v,4,2,1)$-OOC and a $(v,5,2,1)$-OOC has been established of being either the ceiling or the floor of $\frac{v}{8}$, see \cite{MB}, and being either the ceiling or the floor of $\frac{v}{12}$, see \cite{BPW}, respectively. In both cases, the authors say that the code is \emph{optimal} if its size attains the corresponding upper bound.

Motivated by these results, we believe it is natural to consider $(v,5,3,1)$-OOCs and to determine a sharp upper bound on their size. In the next section, after introducing the necessary notation and terminology, we prove that this bound is either the ceiling or the floor of $\frac{v}{10}$ and, consistently with the previous terminology, we say that a $(v,5,3,1)$-OOC is optimal if it achieves this bound. Then, in Section \ref{sec:graph}, we describe optimal codes in terms of suitable graph labellings establishing a connection with difference families over graphs. The fourth section is devoted to direct constructions for optimal and perfect $(v,5,3,1)$-OOCs, while in Section \ref{sec:recursive}, we propose a recursive construction.

\section{A bound on the maximum possible size of a $(v,5,3,1)$-OOC}\label{sec:bound}

In this section, we establish a tight upper bound on the size of a $(v,5,3,1)$-OOC. 
Firstly, we introduce some basic definitions and terminology. Given a subset $C$ of $\mathbb{Z}_v$,
by $\Delta C$ we denote the \emph{list of differences} of $C$, namely the multiset containing the differences $x-y$
for any ordered pair $(x,y)$ of distinct elements of $C$. Using the terminology and notation introduced in \cite{BPW}, it is often useful to consider the underlying set of $\Delta C$, also called the \emph{support}, denoted by $\bar \Delta C$. If $|\bar \Delta C|=\delta$, we say that $C$ is of \emph{type} $t(C)=\delta$. Obviously, if $|C|=k$ then $t(C)\leq |\Delta C|=k(k-1)$. Moreover, we denote by $\mu(C)$ the maximum multiplicity in $\Delta C$, more formally:
$$\mu(C)=\max \{\mu_{\Delta C}(i) \mid 1\leq i \leq v-1\}$$
where by $\mu_{\Delta C}(i)$ we mean the multiplicity of $i$ in $\Delta C$.  \\
Given a set $\mathcal{C}$ of subsets of $\mathbb{Z}_v$, we define the multiset $\bar \Delta \mathcal{C}=\bigcup_{C\in \mathcal{C}}\bar \Delta C$ to be the union of the supports of all subsets of $\mathcal C$.
If for any $C_i,C_j \in\mathcal C$, with $C_i\neq C_j$, we have $\Delta C_i \cap \Delta C_j =\emptyset$, one says that $\mathcal{C}$ is \emph{difference-disjoint}. Clearly, in this case, the multiset $\bar \Delta \mathcal{C}$ is a set. In order to understand if a subset $C$ of $\mathbb{Z}_v$ could be a codeword-set of a $(v,k,\lambda_a,1)$-OOC, we need to 
check the multiplicity of each element of $\mathbb{Z}_v$ in $\Delta C$. 

It is easy to see that, using the above notation, Conditions (\ref{AutoProp}) and (\ref{CrossProp}) can be written in the following equivalent way, which is useful to establish the bound on the size of the code.

Condition (\ref{AutoProp}) is equivalent to requiring that no element of $\mathbb{Z}_v$ appears more than $\lambda_a$ times in the list of differences of any codeword-set, namely:
\begin{gather}
\mu(C_i)\leq \lambda_a\ {\rm for}\  1 \leq i \leq s . \label{AutoPropBis}
\end{gather}

Also, if $\lambda_c=1$, such in the case under investigation, Condition (\ref{CrossProp}) is equivalent to asking that $\mathcal{C}$ is difference-disjoint, namely that 
\begin{gather}
\Delta C_i \cap \Delta C_j =\emptyset\ {\rm for}\ 1 \leq i<j \leq s. \label{CrossPropBis}
\end{gather}

Now, in order to establish the bound, we will investigate the possible types that each $C \in \mathcal{C}$ can take.
Since $k =5$, the multiset $\Delta C$ contains $20$ elements. As $\lambda_a = 3$, we want each element of $\mathbb{Z}_v$ to occur at most $3$ times in $\Delta C$, and therefore, by a basic pigeonhole argument, the type of each set $C$ must be at least $7$.

As already remarked in \cite{BMP}, a codeword-set $C$ of a $(v,k,\lambda_a,1)$-OOC $\mathcal{C}$ is of odd type if and only if $v$
is even and $\frac{v}{2}\in \Delta C$. Therefore, it is immediate that it is not restrictive to assume that a codeword-set of odd type contains $0$ and $\frac{v}{2}$. Furthermore, since the lists of differences of distinct codeword-sets of $\mathcal{C}$ should be disjoint,
$\mathcal{C}$ may have at most one codeword-set of odd type.
Using the same terminology introduced in \cite{BPW}, we say that two subsets $C$ and $C'$ of $\mathbb{Z}_v$ are \emph{equivalent} if we have
$C'=m\cdot C+t$ for a suitable pair $(m,t)$ with $\gcd(m,v)=1$. By means of long and tedious calculations one can check that the following lemma holds. Here we give only a sketch of the proof.

\begin{lemma} \label{lem:type7}
Let $C$ be a $5$-subset of $\mathbb{Z}_v$. Then $\mu(C)\leq 3$ and $|\bar{\Delta} C|=7$ if and only if 
$$v\equiv 0 \pmod 8\ \textrm{and, up to equivalence,} \ C=\left\{0,\frac{v}{8},-\frac{v}{8},\frac{v}{4},\frac{v}{2}\right\}.$$
\end{lemma}
\begin{proof}
Let $C$ be a codeword-set of type $7$ of a $(v,5,3,1)$-OOC. By above remark we can assume $C=\{0,\frac{v}{2},x,y,z\}$. Without loss of generality, we can assume $x\neq -y$. Since $$\Delta C=\pm\left\{\frac{v}{2},x,y,z,\frac{v}{2}+x,\frac{v}{2}+y,\frac{v}{2}+z,x-y,x-z,y-z\right\}$$
since $C$ has type $7$, each non-involution element must occur exactly $3$ times in the multiset $\Delta C$ and the involution element must occur exactly twice in $\Delta C$. Moreover, since $x,y,z$ are all distinct and $x \neq -y$ by assumption,  we consider the following choices for $x$:\\
Case 1. $x=\frac{v}{2}+y$ which implies $x-y = \frac{v}{2}$ which is an immediate contradiction. \\
Case 2. $x=\frac{v}{2}-y$ which eventually implies $x-z = \frac{v}{2}$, leading to a contradiction.\\ 
Case 3. $x=y-x$ which implies $z=\frac{v}{2}+x$, which leads to a contradiction.\\
Case 4. $x=y-z$, then $C=\{0,\frac{v}{2},x,y,y-x\}$ could be of type $7$ if and only if $C=\left\{0,\frac{v}{8},-\frac{v}{8},\frac{v}{4},\frac{v}{2}\right\}$.
\end{proof}

The following lemmas ensure that in a $(v,5,3,1)$-OOC we can have at most two codeword-sets of type 8. 
\begin{lemma}\label{lem:type8a}
    If $C$ is a codeword-set of type $8$ of a $(v,5,3,1)$-OOC, then it is not restrictive to assume that $C$ is a set of the form $C=\{0,a,-a,b,c\}.$
 \end{lemma}
 \begin{proof}
 By way of contradiction, suppose that there exists $C=\{0,x_1,x_2,x_3,x_4\}$ of type $8$ that has no translate of the form $\{0,a,-a,b,c\}$. It is not hard to see that this implies that:
 \begin{itemize}
     \item[(a)] $x_i\neq -x_j$ for every $i\neq j$;
     \item[(b)] $x_i\neq 2x_{j}$ for every $i\neq j$;
     \item[(c)] $x_i-x_j\neq x_k-x_i$ for every $i,j,k$.
 \end{itemize}
 Since $|\bar\Delta C|=8$, by (a) we get that $\bar\Delta C=\pm\{x_1,x_2,x_3,x_4\}$. Hence $x_1-x_3\in \Delta C$ should belong to 
 $\pm\{x_1,x_2,x_3,x_4\}$. Clearly, we cannot have $x_1-x_3\in\{x_1,-x_3\}$ and by (b) we have that $x_1-x_3\not\in\{-x_1,x_3\}$. Hence we can assume that one of the following two cases occurs:\\
 Case 1. $x_1-x_3= x_2$.\\
 If $x_1-x_4=x_2$ then we get $x_3=x_4$, a contradiction. If $x_1-x_4=x_3$, then we get $x_2=x_4$, another contradiction.
 This implies that $x_1-x_4=-x_2$ or $x_1-x_4=-x_3$, in both cases we get a contradiction in view of (c).\\
 Case 2. $x_1-x_3= -x_2$.\\
 From this we have  $x_1-x_4\in \{x_2,x_3,-x_2,-x_3\}$.  If $x_1-x_4=x_2$ we get a contradiction in view of (c). If $x_1-x_4=x_3$ we get a contradiction in view of (a). If $x_1-x_4=-x_2$ then $x_3=x_4$, another contradiction. Finally, if $x_1-x_4=-x_3$ we get a contradiction in view of (c).
 \end{proof}

  The next lemma could be verified keeping in mind the previous one, in order to reduce the  long and tedious calculations.

\begin{lemma} \label{lem:type8}
Let $C$ be a $5$-subset of $\mathbb{Z}_v$. Then $\mu(C)\leq 3$ and $|\bar{\Delta} C|=8$ if and only if one of the following conditions holds: 
\begin{itemize}
\item[(a)] $v \equiv 0 \pmod {10} \ \textrm{and, up to equivalence,} \ C=\left\{0,\frac{v}{10},\frac{v}{5},-\frac{v}{5},\frac{2v}{5}\right\}, $
\item[(b)] $v \equiv 0 \pmod {9} \ \textrm{and, up to equivalence,} \ C=\left\{0,\frac{v}{9},\frac{2v}{9},-\frac{2v}{9},\frac{v}{3}\right\} \textrm{or}\\ C=\left\{0,\frac{v}{9},-\frac{v}{9},-\frac{2v}{9},\frac{v}{3}\right\}, \
\textrm{or}\ C=\left\{0,\frac{v}{9},\frac{2v}{9},-\frac{2v}{9},-\frac{v}{3}\right\} \textrm{or}\ 
C=\left\{0,\frac{v}{9},\frac{2v}{9},\frac{v}{3},-\frac{v}{3}\right\}.$
\end{itemize}
\end{lemma}

We can only have one of the codeword-sets in case \emph{(b)} within our OOC because the sets are not difference-disjoint. Therefore, we can only obtain a $(v,5,3,1)$-OOC with two type $8$ codeword-sets if $v$ is divisible by $90$. Otherwise, we have at most one codeword-set of type $8$.

Our final lemma deals with codeword-sets of type $9$. Also here, keep in mind that, since the type is odd, we can assume that $C$ contains $0$ and $\frac{v}{2}$.

\begin{lemma} \label{lem:type9}
Let $C$ be a $5$-subset of $\mathbb{Z}_v$. Then $\mu(C)\leq 3$ and $|\bar{\Delta} C|=9$ if and only if
\begin{itemize}
\item[(a)]
$v\equiv 0 \pmod {10}\ \textrm{and, up to equivalence,} \ C=\left\{0,\frac{v}{10},-\frac{v}{10},\frac{v}{5},\frac{v}{2}\right\}
\textrm{or}\\ \ C=\left\{0,\frac{v}{10},\frac{v}{5},-\frac{v}{5},\frac{v}{2}\right\}\textrm{or}\
C=\left\{0,\frac{v}{10},\frac{v}{5},\frac{3v}{10},\frac{v}{2}\right\} \textrm{or}\ 
C=\left\{0,\frac{v}{10},\frac{v}{5},\frac{2v}{5},\frac{v}{2}\right\},$
\item[(b)] $v\equiv 0 \pmod {12}\ \textrm{and, up to equivalence,} \ 
C=\left\{0,\frac{v}{6},\frac{v}{4},\frac{v}{3},\frac{v}{2}\right\}
\textrm{or}\ 
C=\{0,\frac{v}{12},-\frac{v}{4}, \frac{5v}{12}, \frac{v}{2}\}.$
\end{itemize}
\end{lemma}

Therefore, we get that, in general, $\mathbb{Z}_v$ has many $5$-subsets $C$ with $\mu(C)\leq 3$ and $|\bar{\Delta} C|= 10$.
All these lemmas are crucial to prove the following result.

\begin{theorem}\label{thm:bound}
If $s$ is the size of a $(v,5,3,1)$-OOC, then we have:
$$s\leq \left\{ \begin{array}{l} 
\lceil{v\over10}\rceil\ \textrm{for}\ \ v\equiv8 \pmod{40} \ \textrm{or}\ \ v\equiv 9 \pmod{90} \ \textrm{or}\ v\equiv 108 \pmod{180} \\ \ \ \ \ \ \ \ \ \ \ \ \textrm{or}\ \ v\equiv216 \pmod{360} \\
\lfloor{v\over10}\rfloor\  \textrm{otherwise}\
\end{array} \right.$$

\end{theorem}
\begin{proof}
Let $\mathcal{C}=\{C_1,\ldots,C_s\}$ be a $(v,5,3,1)$-OOC of size $s$ and let $\delta_i$ be the type of $C_i$.
We have $v-1\geq \delta_1+\ldots+\delta_s$. Let $v=10q+r$ be the Euclidean division of $v$ by $10$, hence
$q=\lfloor{v\over10}\rfloor$ and $0\leq r\leq 9$. We recall that $\mathcal{C}$ may have at most one codeword-set of odd type, hence at most one of type $7$ or $9$. Additionally, by Lemma \ref{lem:type8} there are at most two codeword-sets of type $8$. Therefore, there are at most $3$ codeword-sets with type less than $10$. This implies $v-1\geq 10(s-3) + 23 $, namely
$10q+r+6\geq 10s$, which implies $s\leq q+1$.

From now on, we will assume that the size is as large as possible, namely $s=q+1$. In order to prove the thesis, it is sufficient 
to show that either $v\equiv8 \pmod{40}$ or $v\equiv 9\pmod {90}$ or $v\equiv 108 \pmod{180}$ or $v\equiv216\pmod{360}$. 

We split the proof into four cases.
\begin{itemize}
\item[Case 1:] $v\equiv 0\pmod{10}$.\\
We have at most one codeword-set of odd type and at most two codeword-sets of type $8$.\\
Case 1a. Suppose that a codeword-set of type $7$ exists. As there is at most one codeword-set of odd type, there is no codeword-set of type $9$.
Then  $v-1\geq 7+8+8+10(s-3)$. Recalling that $v=10q+r$ and $s=q+1$, we obtain $r\geq 4$, against the hypothesis that $r=0$.\\
Case 1b. Suppose now that a codeword-set of type $7$ does not exist.
Then  $v-1\geq 9+8+8+10(s-3)$,  which, as before, leads to a contradiction.
\item[Case 2:] $v\not\equiv 0\pmod{9}$ and $v\not\equiv 0\pmod{10}$.\\
By previous lemmas, we have at most one codeword-set of type $7$, no codeword-sets of type $8$ and at most one codeword-set of type $9$.\\
Case 2a. Suppose that a codeword-set of type $7$ exists. As there is at most one codeword-set of odd type, there is no codeword-set of type $9$.
Then $v-1\geq 7 +10(s-1)$, which implies $r\geq 8$; recalling that $v$ is even, we get
$r=8$. By Lemma \ref{lem:type7} we have $v\equiv 0 \pmod{8}$ and so we must have $v\equiv 8 \pmod{40}$.\\
Case 2b. Suppose now that a codeword-set of type $7$ does not exist.
Then $v-1\geq 9 + 10(s-1)$, which implies $r\geq 10$, a contradiction.
\item[Case 3:] $v\equiv 0\pmod{9}$, $v\not\equiv 0\pmod{10}$ and $v$ odd.\\
Since $v$ is odd, there is no codeword-set of odd type. By Lemma \ref{lem:type8}, there is at most one codeword-set of type $8$.
Then $v-1\geq 8+10(s-1)$, which implies $r\geq 9$, that is $r=9$. Hence $v\equiv 9\pmod{90}$. 
\item[Case 4:] $v\equiv 0\pmod{9}$, $v\not\equiv 0\pmod{10}$ and $v$ even.\\
We have at most one codeword-set of type odd and one codeword-set of type $8$. \\
Case 4a. Suppose that a codeword-set of type $7$ exists. As there is at most one codeword-set of odd type, there is no codeword-set of type $9$.
Then  $v-1\geq 7+8+10(s-2)$, which implies $r\geq 6$, which means $r=6,8$.
If $r=6$, then we have $v\equiv 6\pmod{10}$, $ v\equiv 0\pmod 8$ (by Lemma \ref{lem:type7}) and $v\equiv 0\pmod 9$. Therefore, using the Chinese Remainder Theorem, we have $v\equiv 216 \pmod {360}$. Analogously, if $r=8$, then we have $v\equiv 8\pmod{10}$, $ v\equiv 0\pmod 8$ and $v\equiv 0\pmod 9$, which simultaneously hold only for $v\equiv 288 \pmod {360}$. Note that this class of values of $v$ is included in $v\equiv 8\pmod{40}$ that results from Case 2a.\\
Case 4b. Suppose now that a codeword-set of type $7$ does not exist.
Then  $v-1\geq 8+9+10(s-2)$, which implies $r\geq 8$. As $v$ is even, we must have $r=8$, that is $v\equiv 8 \pmod {10}$. Additionally, we have $v\equiv 0 \pmod {9}$, and $v\equiv 0 \pmod {12}$ by Lemma \ref{lem:type9}. Therefore, using the Chinese Remainder Theorem, we have $v\equiv 108 \pmod {180}$. \qedhere
\end{itemize}
\end{proof}
As already remarked, in general, $\mathbb{Z}_v$ has many $5$-subsets $C$ with $\mu(C)\leq 3$ and $|\bar{\Delta} C|= 10$. In the following proposition we establish all codeword-sets of type $10$. As usual, given $x \in \mathbb{Z}_v$, by $ord(x)$ we denote the order of $x$ in the group.

\begin{proposition}\label{prop:tipo10}
     Let $C$ be a $5$-subset of $\mathbb{Z}_v$. Then $\mu(C)\leq 3$ and $|\bar{\Delta} C|=10$ if and only if
    \begin{itemize}
          \item[(a)] $C=\{0, x,-x, 3x, -3x\}$ where $ord(x)\geq 11$ and $ord(x) \neq 12$, up to equivalence;
                \item[(b)]  $C=\{0, x,2x, -2x, 3x\}$  where $ord(x)\geq 11$, up to equivalence;
           \item[(c)]      $C=\{0, x,2x, 4x, 5x\}$  where $ord(x)\geq 11$, up to equivalence;
        \item[(d)] $v\equiv 0 \pmod 3$,  $ord(x) \not\in \{1,2,3,4,6,9\}$, and, up to equivalence, $C=\{0,\frac{v}{3}, \frac{2v}{3}, x,-x\}$;
        \item[(e)]  $v\equiv 0 \pmod{11}$ and, up to equivalence, $ C= \{0,\frac{v}{11},\frac{2v}{11},\frac{3v}{11},\frac{6v}{11}\}$ or $C=\{0,\frac{v}{11},\frac{2v}{11},\frac{4v}{11},\frac{7v}{11}\}$;
        \item[(f)] $v\equiv 0\pmod{12}$, and, up to equivalence, $C=\{0,\frac{v}{12},\frac{v}{6},\frac{v}{3},\frac{3v}{4}\}$.
    \end{itemize}
\end{proposition}
\begin{proof}
(a) We have $\bar\Delta C = \pm\{x,2x,3x,4x,6x\}$. In order to have a type $10$ codeword-set these quantities must be pairwise distinct: assume that $ix = \pm jx$ for some $i,j \in \{1,2,3,4,6\}$; then $ord(x) \mid (i-j)$ or $ord(x) \mid (i+j)$. 
So, a codeword-set of type $10$ implies $ord(x) >10$. Since $ord(x) = 12$ would imply that $6x$ is an involution, we also require $ord(x) \neq 12$.
The codeword-set $C$ is then of type $10$ whenever $ord(x) \geq 11$, and $ord(x) \neq 12$.

(b) and (c) We have $\bar\Delta C = \pm\{x,2x,3x,4x,5x\}$; as for case (a), these differences are not distinct for $ord(x) \mid (i-j)$ or $ord(x) \mid (i+j)$ with $i,j \in \{1,2,3,4,5\}$, hence $ord(x) \geq 10$. As for case (a),  $ord(x) = 10$ would imply that $5x$ is an involution, so the statement follows for $ord(x)\geq 11$.

(d)  We have $\bar\Delta C=\pm\{\frac{v}{3},x+\frac{v}{3},x-\frac{v}{3}, x,2x\}$. A direct examination of the possible collisions shows that these differences are distinct for:
\[
2x,3x,4x,6x \neq 0, \qquad x,2x,3x \neq \pm \frac{v}{3}. 
\]
An easy check shows that these conditions are verified for $ord(x) \not\in \{4,6,9\}$. The condition $ord(x) \not \in \{1,2,3\}$ guarantees that the elements in $C$ are distinct.

(e) and (f) By direct computation we get respectively: 
\[
\begin{aligned}&\bar\Delta C = \pm\left\{\frac{v}{11},\frac{2v}{11},\frac{3v}{11}, \frac{4v}{11}, \frac{5v}{11}\right\},\\
&\bar\Delta C = \pm \left\{ \frac{v}{12}, \frac{v}{6}, \frac{v}{4}, \frac{v}{3}, \frac{5v}{12}\right\}.
\end{aligned}
  \]

We now show the proof for the other direction.
In order to have a type $10$ codeword-set $C$ with autocorrelation at most $3$,  the multiplicities of the elements in $\bar\Delta C$ must form an integer partition of $10$ having exactly $5$ terms, where each multiplicity is not greater than $3$. A direct check shows that the partitions of $10$ satisfying these properties are
\begin{equation}\label{eq:partitions}
(3,3,2,1,1) \quad (3,2,2,2,1) \quad(2,2,2,2,2).
\end{equation}
In order to classify admissible codeword-sets, we construct all the vertex labelings of the complete graph $K_5$ whose induced set of differences, up to sign, have multiplicities within one of the integer partitions of \eqref{eq:partitions}.
Since this is a quite laborious check, here we consider in detail the case where an edge label has multiplicity $3$.

Assume that there is a difference $a$ with multiplicity $3$. It can be easily seen that the set of edges with label $\pm a$ is either a cycle of length $3$, a path of length $3$ or a path of length $2$ with an independent edge. In what follows, we examine each case.

If the edges with  label $\pm a$ form a cycle of the length three, then $a=\frac{v}{3}$ and the codeword-set has the form $\{0,a,2a,b,c\}=\{0,a,-a,b,c\}$ for some $b,c \in \mathbb{Z}_v$. We have:
\[
\bar\Delta \{0,a,2a,b,c\} = \pm \{a,b,c,b-a,b+a,c-a,c+a,b-c\}.
\]
It can be easily seen that we may assume, up to relabelling, that $b-a$ is an edge label with multiplicity $2$ or $3$. By direct computations, we have that the admissible equalities are:
\[
\begin{array}{l|l}
\text{Equality}& \text{Codeword-set} \\ \hline
b-a = -c & C^1 = \{0,a,2a,b,a-b\} \\
b-a = -c+a & C^2 = \{0,a,2a,b,2a-b\} \\
b-a = -c-a & C^3 = \{0,a,2a,b,-b\} \\
b-a = -b+c & C^4 = \{0,a,2a,b,2b-a\} \\
\end{array}
\]
The codeword-sets $C^1$ and $C^2$ are equivalent to $C^3$. $C^3$ is of type $10$ for $4b,6b,9b \neq 0$, hence constructing the family of codeword-sets of point (d). A case-by-case examination shows that $C^4$ cannot be completed to a type $10$ codeword-set for any choice of $b$.

If the edges with  label $\pm a$ form a path of length three, then the codeword-set is $\{0,a,2a,3a,b\}$ for some $b \in \mathbb{Z}_v$.  We have:
\[
\bar\Delta \{0,a,2a,3a,b\} = \pm \{a,2a,3a,b-a,b-2a,b-3a,b\}.
\]
If $3a$ has multiplicity greater than $1$, either $3a = -2a$, or at least one of the edges joining $b$ with $\{0,a,2a,3a\}$ has label $3a$. In the former case, $5a=0$, and it follows that $b=4a$, gaining a codeword-set not of type $10$. In the latter, $b\in \{-3a,-2a,5a,6a\}$: the cases $b=-2a$ and $b=5a$ are equivalent, giving the admissible codeword-set $\{0,a,2a,-2a,3a\}$, while $b=-3a$ and $b=6a$ are equivalent, and they give an admissible codeword-set when $ord(a)=11$, that is up to equivalence $\{0,\frac{v}{11},\frac{2v}{11},\frac{3v}{11},\frac{6v}{11}\}$.
Assume now that $3a$ has multiplicity $1$. It can be directly verified that no further edges may receive the label $2a$, and the resulting integer partition is $(3,2,2,2,1)$. A case-by-case check shows that the unique admissible codeword-set has $2b=3a$: setting $x = b-a$ results in the codeword-set $\{0,x,-x,3x,-3x\}$, which is the family of point (a).

If the edges with the label $\pm a$ form a path of length $2$ with an independent edge, the codeword-set is $\{0,a,2a,a+b,2a+b\}$ for some $b \in \mathbb{Z}_v$, and:
\[
\bar\Delta\{0,a,2a,a+b,2a+b\} = \pm \{a,b,a+b,2a,2a+b,a-b\}.
\]
When examining the possible collisions, we have the following admissible cases:
\[
\begin{aligned}
2a = -2a-b,& \qquad 2a = b-a, \\
2a+b = -b-a, &\qquad a-b = b.\\
\end{aligned}
\]
If $b=-4a$ or $b = 3a$ we obtain, up to equivalence, the codeword-set $\{0, a,2a,4a,5a\}$, that has type $10$ when $ord(a)\geq 11$. 
If $2b = -3a$ or $a = 2b$ the resulting codeword-set is $\{0,b,2b,-2b,3b\}$, that is in case (b).

One can reason in a similar way when each difference has multiplicity two in $K_5$.
\end{proof}

\subsection{Optimal $(v,5,3,1)$-OOC of small order}

Following the terminology used in several papers, for instance \cite{BMP, BPW,MB,WC},  we will say that a $(v,5,3,1)$-OOC $\mathcal{C}$ is \emph{optimal} when its size reaches the upper bound given by Theorem \ref{thm:bound}. We will also say that $\mathcal{C}$ is \emph{perfect} when it is optimal and the only missing difference in $\Delta \mathcal{C}$ is zero. We have to point out that in \cite{BT16} the authors use the term optimal with a different meaning. Indeed, they say that a $(v,k,\lambda_a,\lambda_c)$-OOC is optimal when it has maximum size. Note that an optimal code according to our definition is also an optimal code according to the definition given in \cite{BT16}, while the converse in general is not true. 
Looking at Table 1 of \cite{BT16}, in view of the bound established in Theorem \ref{thm:bound}, we can establish that there exists an optimal $(v,5,3,1)$-OCC with $29\leq v\leq 100$ if and only if $v=29,37,39,49,55,58,59,65,67,68,69,75,77,78,79,84,87,89,95,97,98,99$. For all other values of $v$ with $29\leq v\leq 100$ there exists a $(v,5,3,1)$-OOC with fails optimality by missing just one codeword-set.
Other optimal $(v,5,3,1)$-OOCs of small order can be found noticing that for some values of $v$ the blocks of an optimal $(v,5,2,1)$-OOCs can also be viewed as the blocks of an optimal $(v,5,3,1)$-OOC. Into details, the following holds.
\begin{proposition}
Let $v\in\{12,13,14,15,16,17,18,19,29,37,39\}$, then
there exists an optimal  $(v,5,3,1)$-OOC. 
\end{proposition}
\begin{proof}
Let $v$ be as in the statement, in \cite{BPW} (see Table 1) the authors prove that an optimal $(v,5,2,1)$-OOC exists for these values of $v$.
Also, by Theorem 2.1 of \cite{BPW}, for these values of $v$,  an upper bound for the size of a $(v,5,2,1)$ is $\lfloor{\frac{v}{12}}\rfloor$, and this is the size of the optimal $(v,5,2,1)$-OOC constructed.
By the hypothesis on $v$, we have that $\lfloor{\frac{v}{10}}\rfloor=\lfloor{\frac{v}{12}}\rfloor$ and hence, in view of Theorem \ref{thm:bound}, the blocks of the optimal $(v,5,2,1)$-OOCs constructed in \cite{BPW} are also the blocks of optimal $(v,5,3,1)$-OOCs.
\end{proof}

\begin{example}
It is easy to check that $C_1=\{0,1,-1,6,-6\}$, $C_2=\{0,8,-8,11,-11\}$ and $C_3=\{0,10,-10,14,-14\}$ are the blocks of a perfect $(37,5,2,1)$-OOC. Since a $(37,5,3,1)$-OOC with four blocks cannot exist by Theorem \ref{thm:bound}, we have that $\{C_1,C_2,C_3\}$ is also a perfect $(37,5,3,1)$-OOC.
\end{example}



\section{Cycle/bull/tadpole difference families and $(v,5,3,1)$-OOCs}\label{sec:graph}

In \cite{MB} the authors show that in most cases an optimal $(v,4,2,1)$-OOC can be described in terms of a suitable graph labeling, establishing, in this way, a connection between OOCs and cyclic decompositions of a complete multipartite graph. This idea has been further developed in \cite{BMP}. In this section we describe optimal $(v,5,3,1)$-OOCs in terms of suitable \emph{cycle labelings} or, equivalently, of \emph{bull labelings} or even in terms of \emph{tadpole labelings}.

\begin{figure}[ht]
    \centering 
     \includegraphics[width=0.3\textwidth]{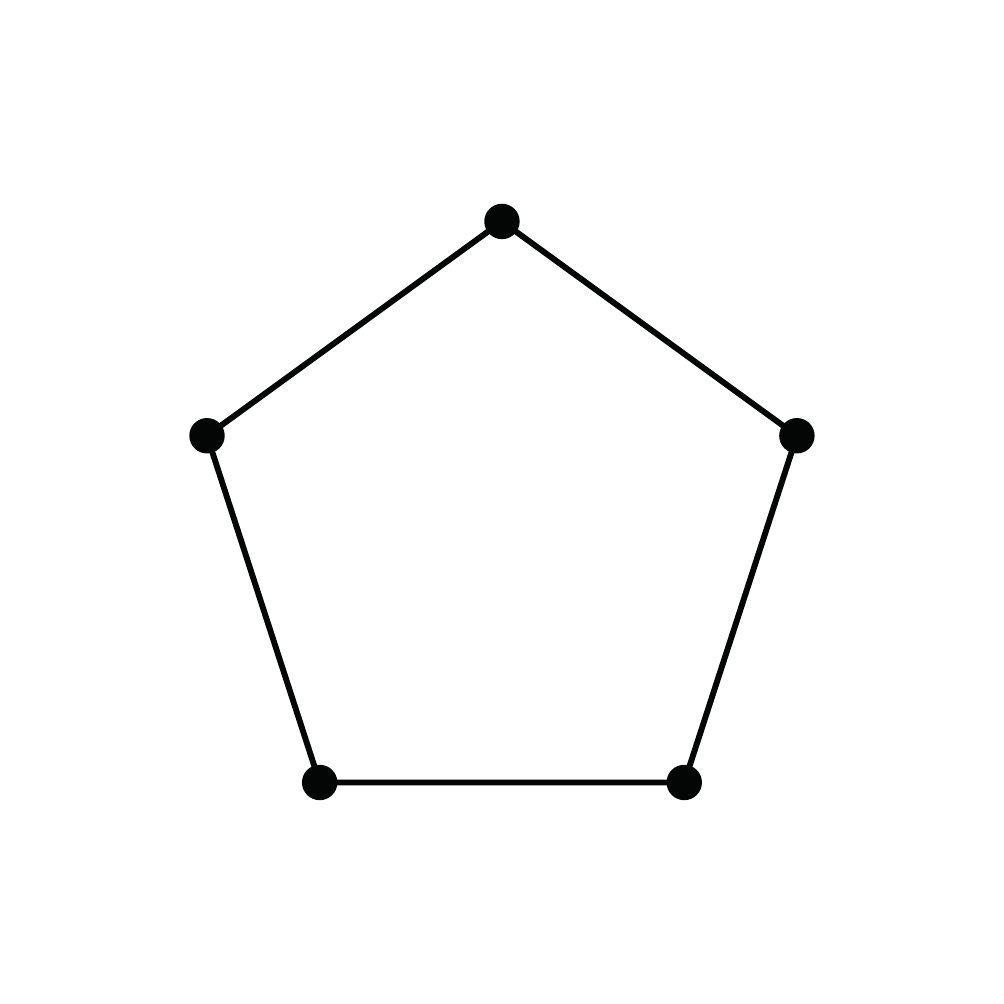}\quad 
    \includegraphics[width=0.3\textwidth]{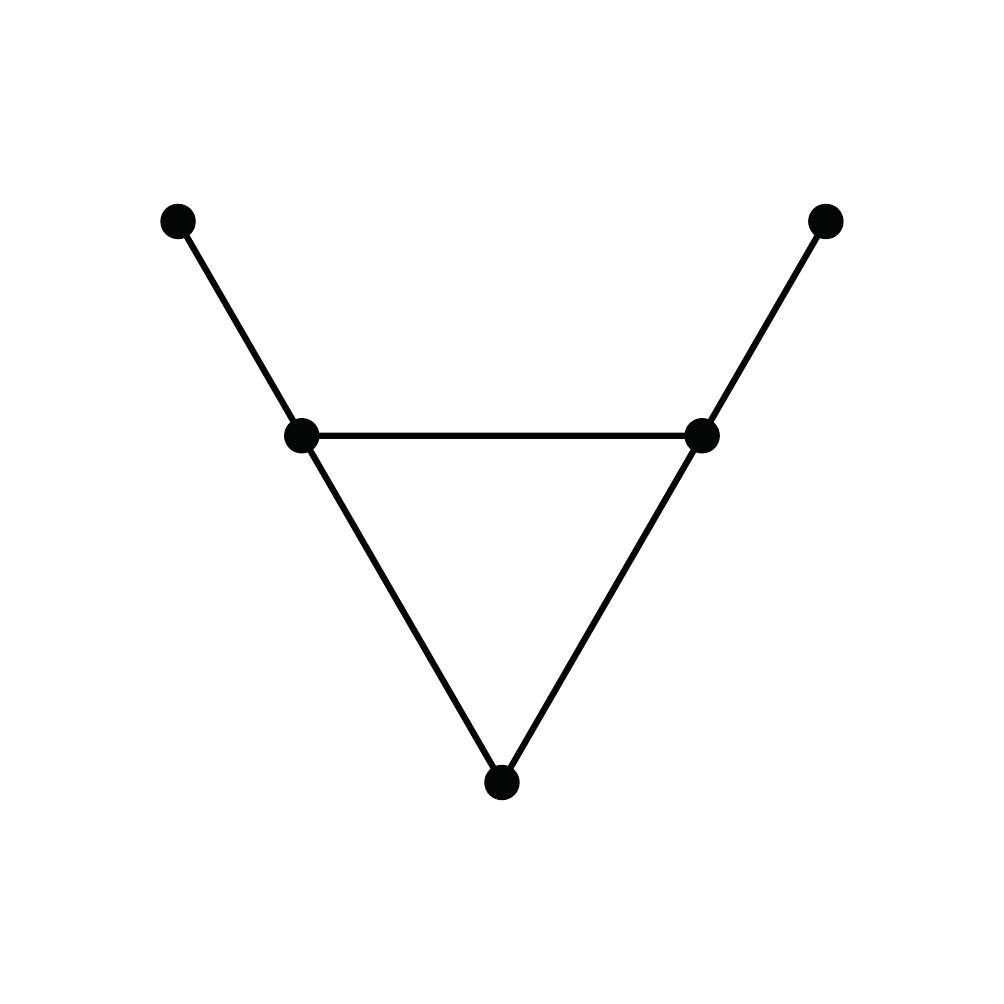} 
      \caption{The cycle of length five $C_5$ and the bull graph $G$.}
    \label{fig:C5B}  
\end{figure}

\begin{figure}[ht]
    \centering 
     \includegraphics[width=0.4\textwidth]{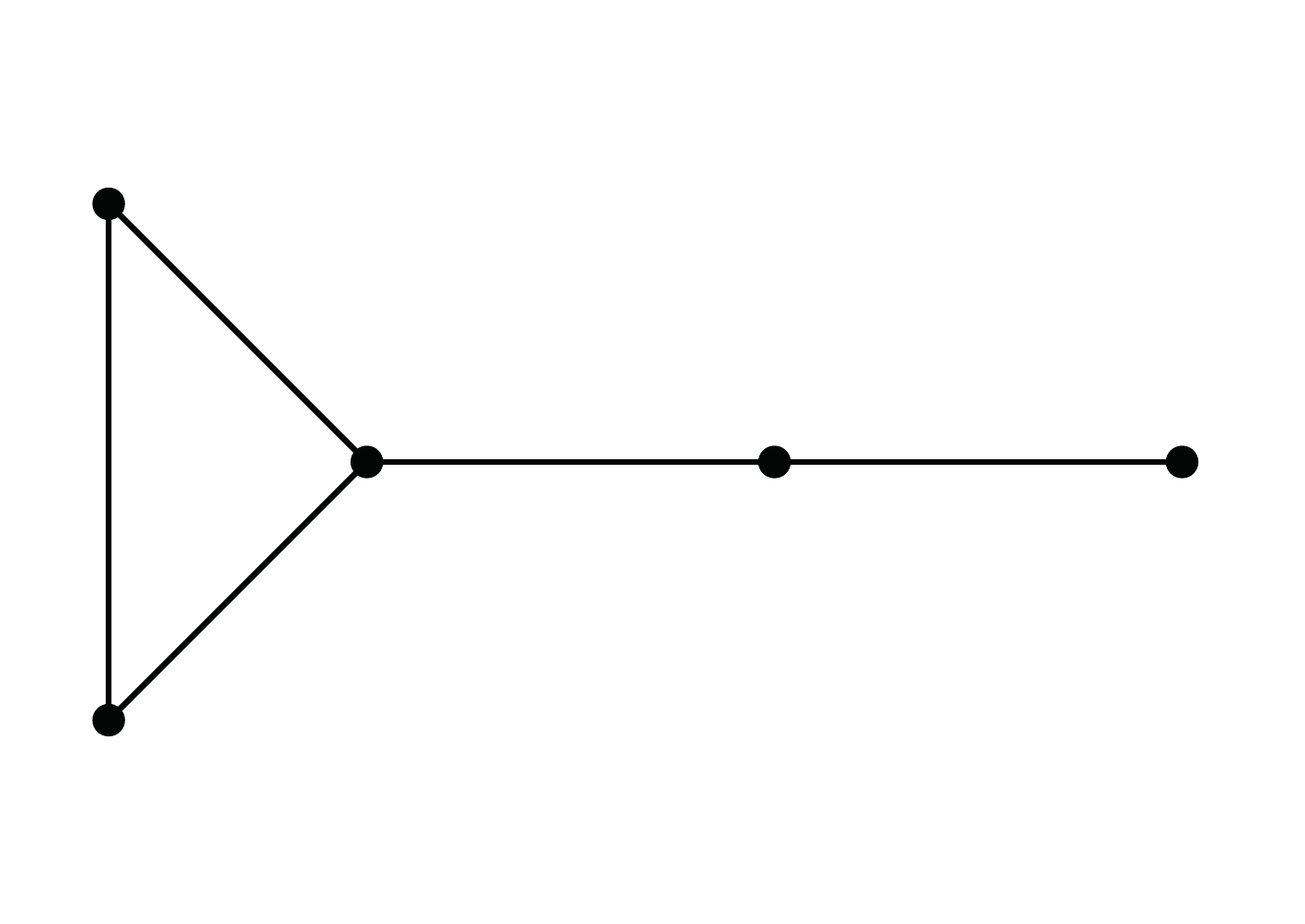}\quad 
    \includegraphics[width=0.4\textwidth]{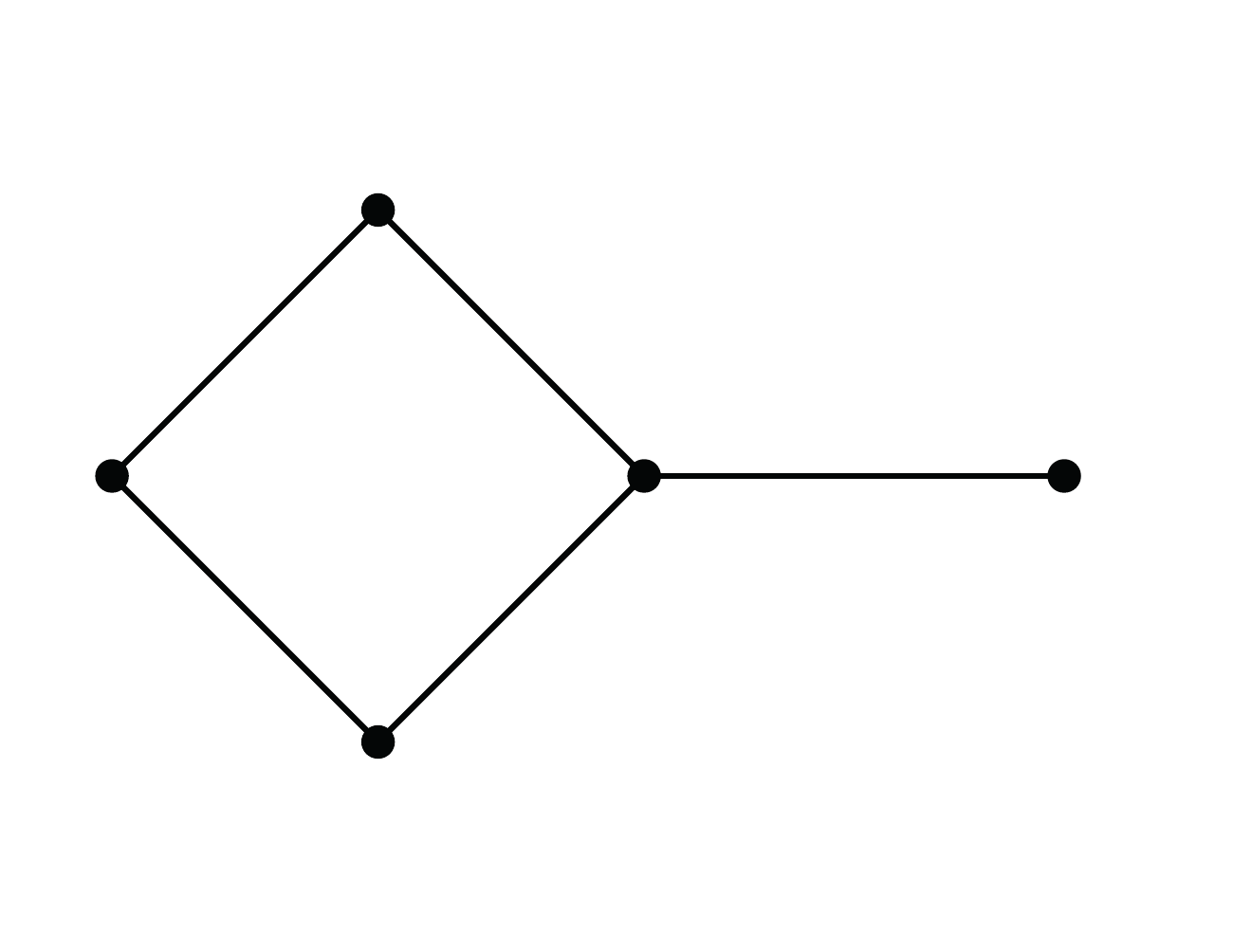} 
    \caption{The tadpole graphs $T_{3,2}$ and $T_{4,1}$.}
    \label{fig:tadpole}
\end{figure}

The graphs $C_5$ and $G$ of Figure \ref{fig:C5B} are the cycle of length $5$ and the so called \emph{bull graph}, respectively. Also, we recall that the $(m,n)$-\emph{tadpole graph}, denoted by $T_{m,n}$, is the graph consisting of a cycle of length $m\geq3$ and a path on $n$ vertices connected by a bridge. $T_{3,2}$ and $T_{4,1}$ are represented in Figure \ref{fig:tadpole}. We say that a cycle $C_5$ with vertices in $\mathbb{Z}_v$ is \emph{good} if it is of the form shown in Figure \ref{fig:C5xGx} for some element $x\in \mathbb{Z}_v$. We will denote such a good cycle by $C_5^x$. Analogously, we say that the bull graph with vertices in $\mathbb{Z}_v$ is \emph{good} if it is of the form shown in Figure \ref{fig:C5xGx} for some $x\in \mathbb{Z}_v$ and we will denote it by $G^x$. Finally, we say that the tadpole graphs $T_{3,2}$ and $T_{4,1}$ with vertices in $\mathbb{Z}_v$ are \emph{good} if they are of the form shown in Figure \ref{fig:tadpolex} for some $x\in \mathbb{Z}_v$ and we will denote them by $T_{3,2}^x$ and $T_{4,1}^x$, respectively. 

\begin{figure}[ht]
    \centering 
     \includegraphics[width=0.3\textwidth]{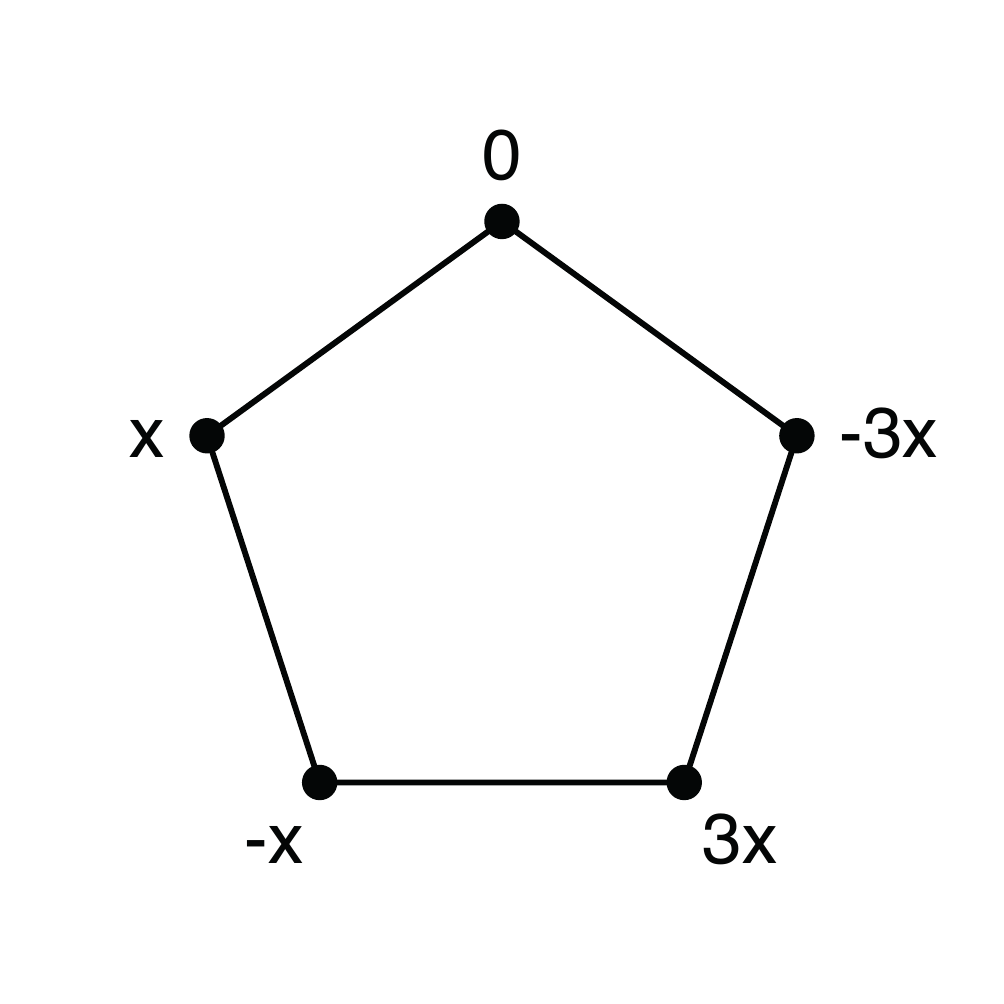}\quad\quad
    \includegraphics[width=0.3\textwidth]{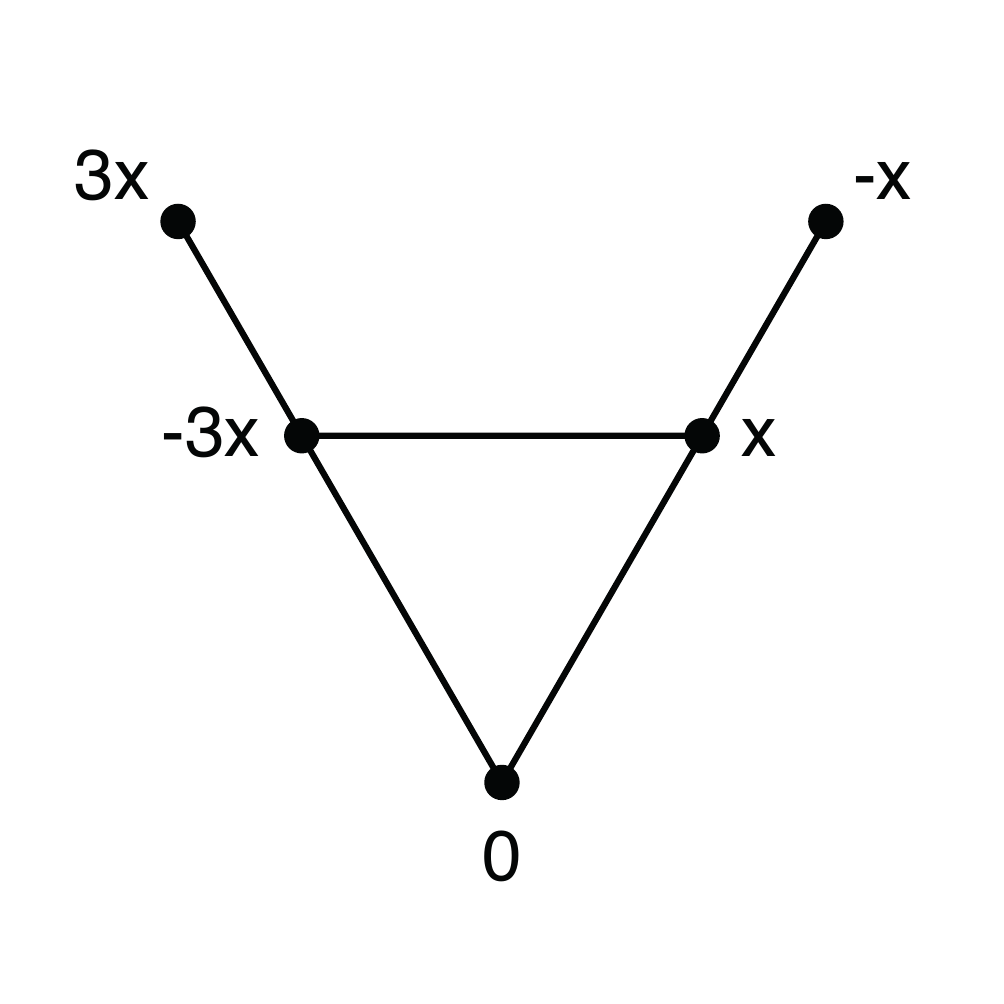}
    \caption{The good cycle $C_5^x$ and the good bull graph $G^x$.}
    \label{fig:C5xGx}
\end{figure}

\begin{figure}[ht]
    \centering 
     \includegraphics[width=0.4\textwidth]{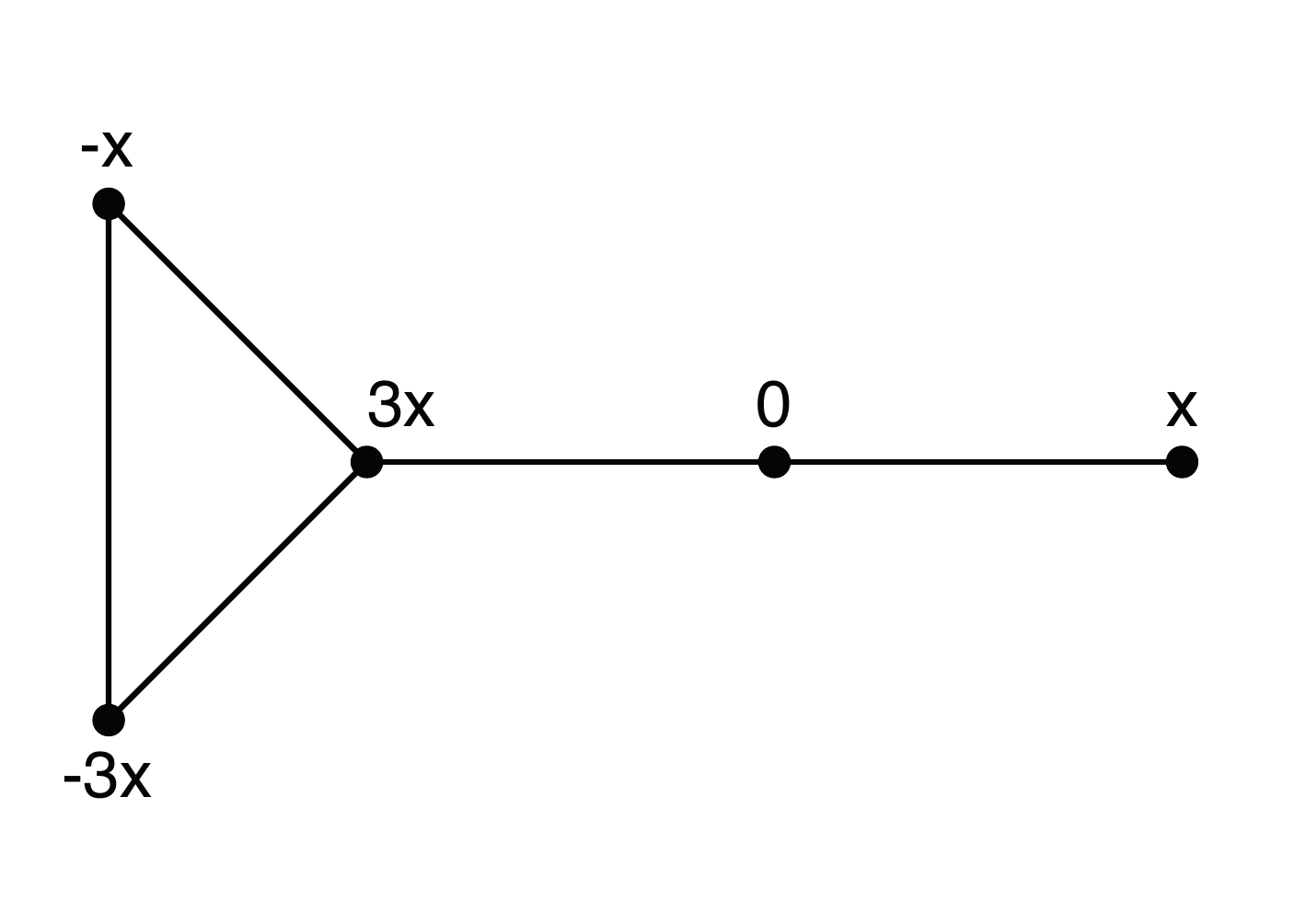}\quad 
    \includegraphics[width=0.4\textwidth]{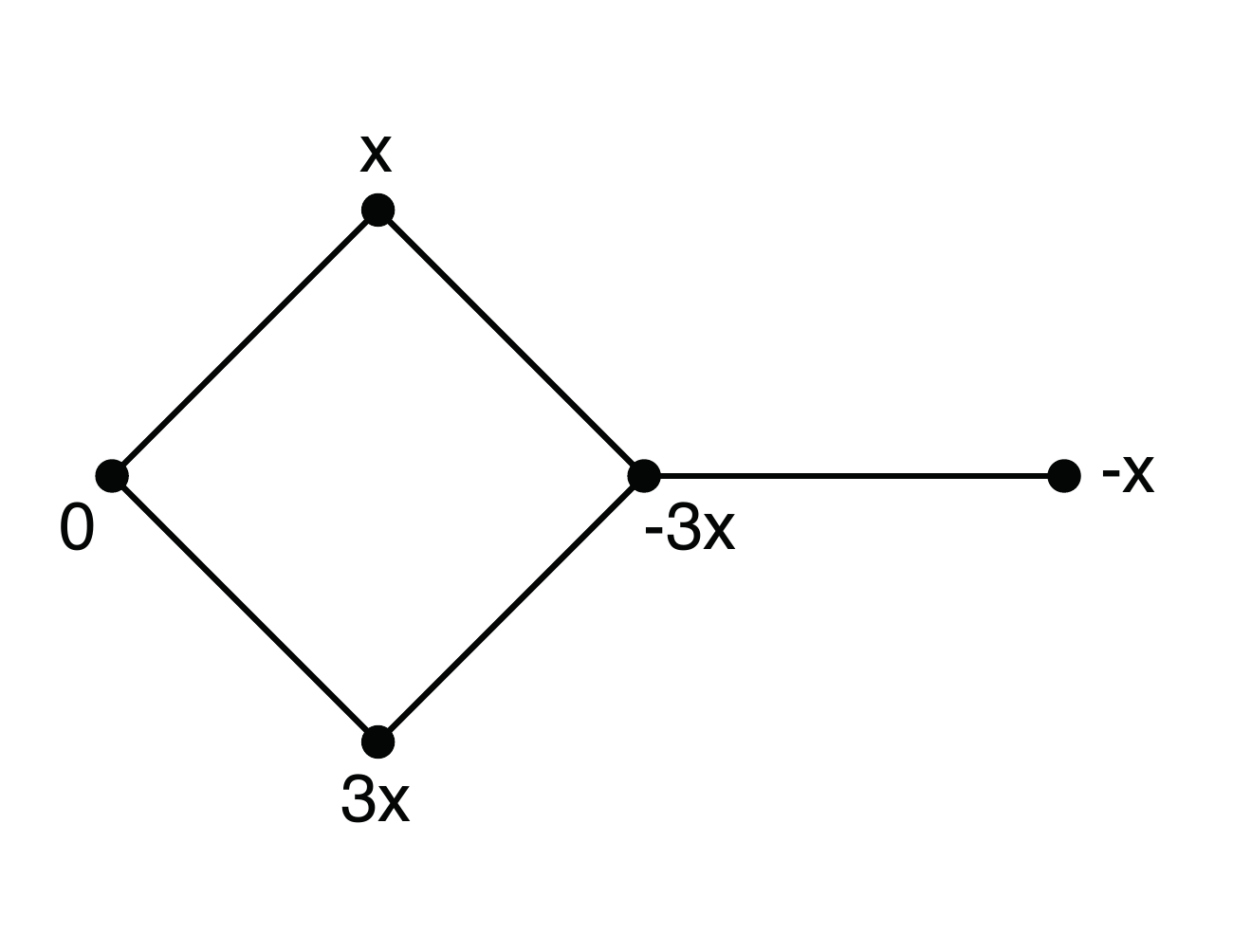} 
    \caption{The  good tadpole graphs $T_{3,2}^x$ and $T_{4,1}^x$.}
    \label{fig:tadpolex}
\end{figure}

Now we recall the concept of a difference family over graphs, for further details see \cite{BP}. Given a simple graph $\Gamma$ with vertices in $\mathbb{Z}_v$, the \emph{list of differences} of $\Gamma$ is the multiset $\Delta\Gamma$ containing all possible differences $x-y$ with $(x,y)$ an ordered pair of adjacent vertices of $\Gamma$. Given a graph $\Gamma$ and two positive integers $m$ and $v$, a $(mv,m,\Gamma,1)$ \emph{difference family} (DF, in brief) is a collection of graphs, called blocks, all isomorphic to $\Gamma$, with vertices in $\mathbb{Z}_{mv}$ whose list of differences contains exactly once each element of $\mathbb{Z}_{mv}\setminus \{0,v,2v,\ldots,(m-1)v\}$.
When $m=1$ we find the concept of a $(v,\Gamma,1)$-DF, say $\mathcal{F}$. Since in this case $\Delta \mathcal{F}=\mathbb{Z}_v\setminus \{0\}$, a necessary condition for its existence is $v\equiv 1\pmod{2|E(\Gamma)|}$.
Also, if $m=1$ and $\Gamma$ is the complete graph of order $k$ we find again the classical concept of a $(v,k,1)$-DF, see \cite{AB}.

We will say that a $(mv,m,\Gamma,1)$-DF, where $\Gamma$ is the cycle of length $5$, or the bull graph or the tadpole graph $T_{3,2}$ or $T_{4,1}$, is \emph{good} if all its blocks are good. Note that if the $5$-subset $C_x=\{0,x,-x,3x,-3x\}$ is of type $10$, then the underlying set of its list of differences coincides with the list of differences of the good cycle $C_5^x$, along with that of the good bull graph $G^x$ and with that of the good tadpole graphs $T_{3,2}^x$ and $T_{4,1}^x$. In other words, we have
$\bar\Delta C_x=\Delta C_5^x=\Delta G^x=\Delta T_{3,2}^x=\Delta T_{4,1}^x$.
It follows that a good $(mv,m,C_5,1)$-DF $\{C_5^{x_1},\ldots, C_5^{x_t}\}$ is equivalent to a good $(mv,m,G,1)$-DF $\{G^{x_1},\ldots, G^{x_t}\}$, to a good $(mv,m,T_{3,2},1)$-DF $\{T_{3,2}^{x_1},\ldots, T_{3,2}^{x_t}\}$ and to a good $(mv,m,T_{4,1},1)$-DF $\{T_{4,1}^{x_1},\ldots, T_{4,1}^{x_t}\}$, and that it gives a $(mv,5,3,1)$-OOC whose codeword-sets are $C_{x_1},\ldots,C_{x_t}$.

In particular the following holds.
\begin{proposition}
Suppose there exists a good $(mv,m,\Gamma,1)$-DF with $1\leq m\leq 5$ or $m=7$, and with $\Gamma=C_5^x$ or $G^x$ or $T_{3,2}^x$ or $T_{4,1}^x$. Then there exists an optimal $(mv,5,3,1)$-OOC.
\end{proposition}
\begin{proof}
Let $\mathcal{C}$ be the $(mv,5,3,1)$-OOC associated with a $(mv,m,\Gamma,1)$-DF with $\Gamma$ as in the statement. Then $\mathcal{C}$ has size $q=\frac{mv-m}{10}$, which implies $mv=10q+m$. Since by assumption $m$ does not exceed $9$,
we have that $q$ and $m$ are nothing but the quotient and the remainder of the Euclidean division of $mv$ by $10$. Hence $q=\lfloor \frac{mv}{10} \rfloor$. So, in view of Theorem \ref{thm:bound}, to prove that the code $\mathcal{C}$ is optimal and we have to exclude that one of the following cases can occur:
\begin{itemize}
    \item[(1)] $mv\equiv 8\pmod{40}$
    \item[(2)] $mv\equiv 9\pmod{90}$
    \item[(3)] $mv\equiv 108\pmod{180}$   
    \item[(4)] $mv\equiv 216\pmod{360}$
\end{itemize}
Recalling that $mv=10q+m$, we have that Cases (1) and (3) would imply $m=8$, while Cases (2) and (4) would imply $m=9$ and $m=6$, respectively. In any case we contradict the assumption on $m$, hence the code is optimal.
\end{proof}

\begin{corollary}
    If there exists a good $(v,\Gamma,1)$-DF with $\Gamma=C_5^x$ or $G^x$ or $T_{3,2}^x$ or $T_{4,1}^x$, then there exists a perfect $(v,5,3,1)$-OOC.
\end{corollary}

Clearly, one can do the same reasoning with the other codeword-sets of type $10$. For example if $C=\{0,x,2x,-2x,3x\}$ one can check that the associated graphs are the ones described in Figure \ref{fig:2tipo10}.

\begin{figure}[ht]
    \centering 
     \includegraphics[width=0.23\textwidth]{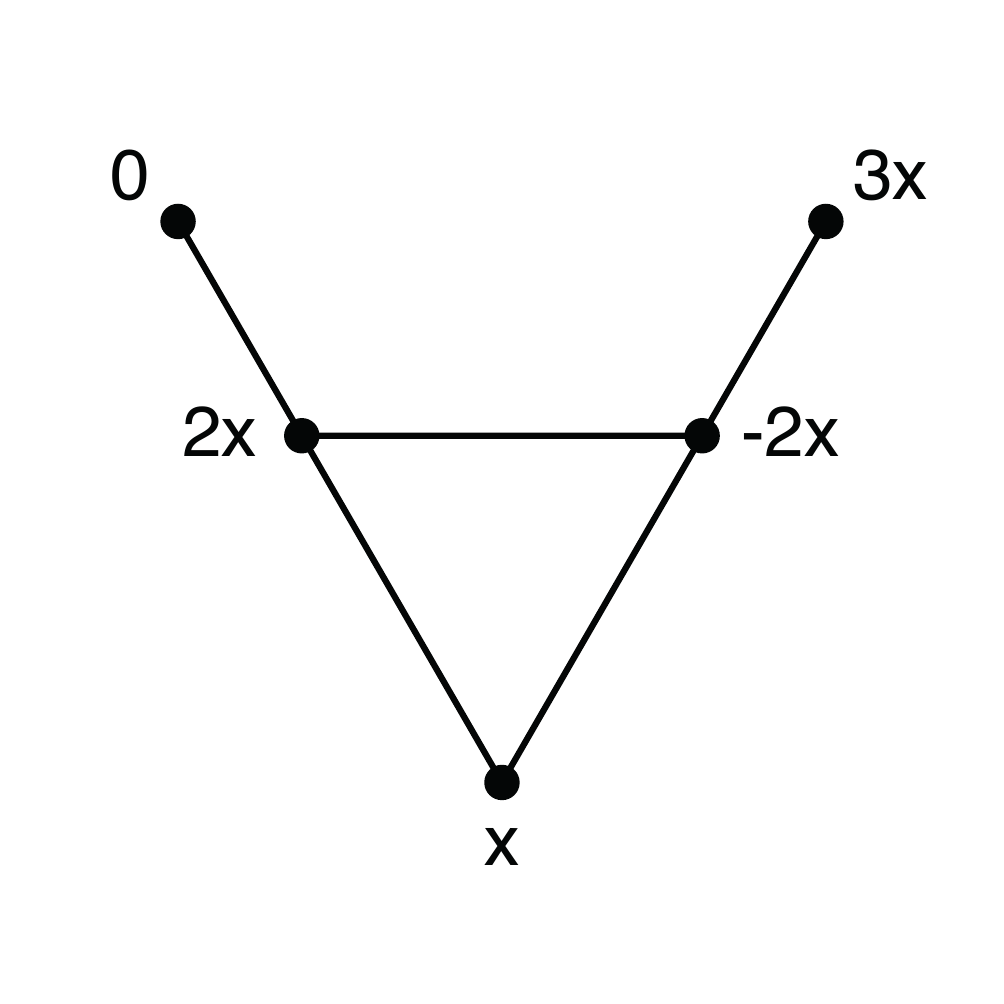}\quad 
    \includegraphics[width=0.32\textwidth]{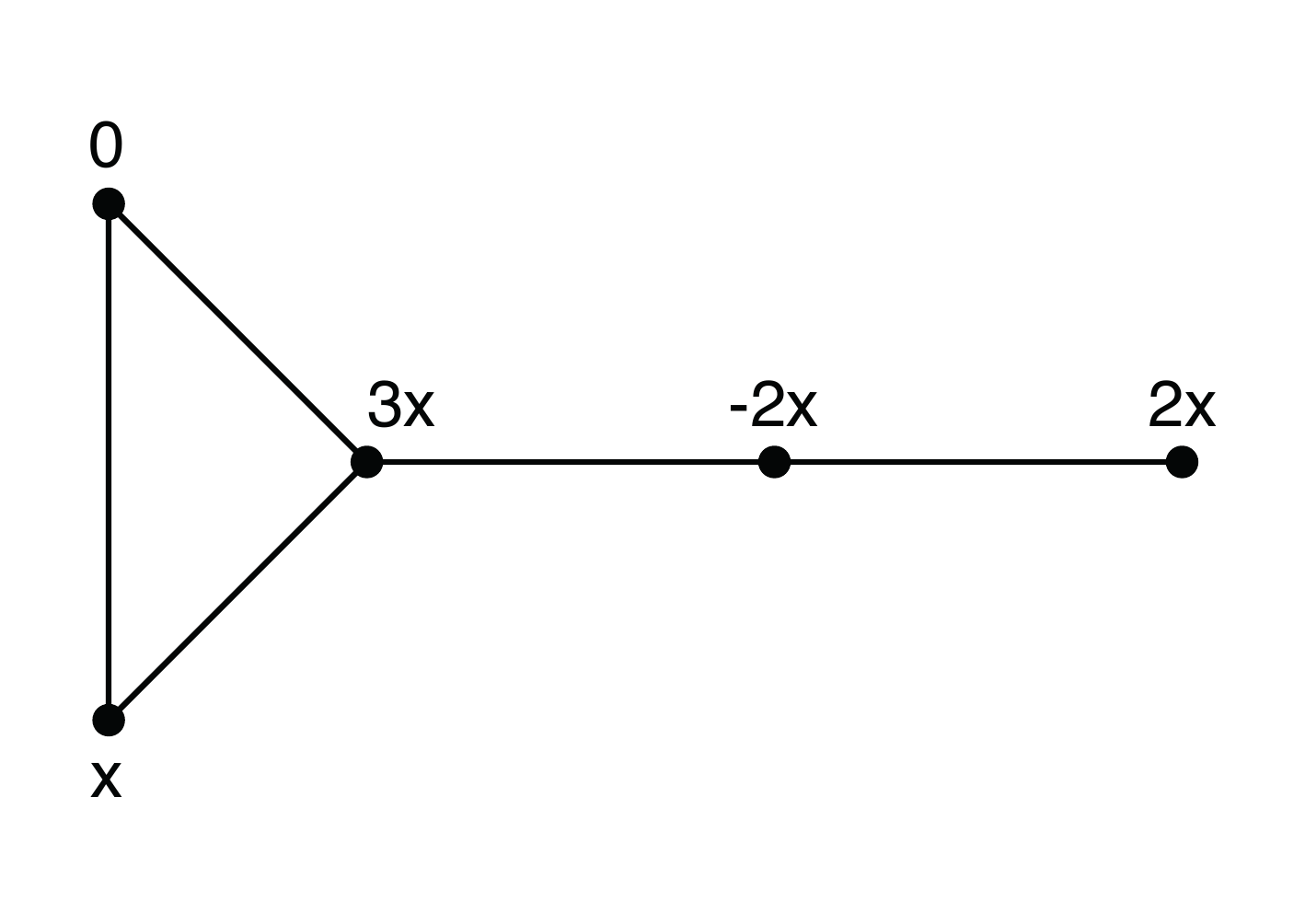} \quad
        \includegraphics[width=0.33\textwidth]{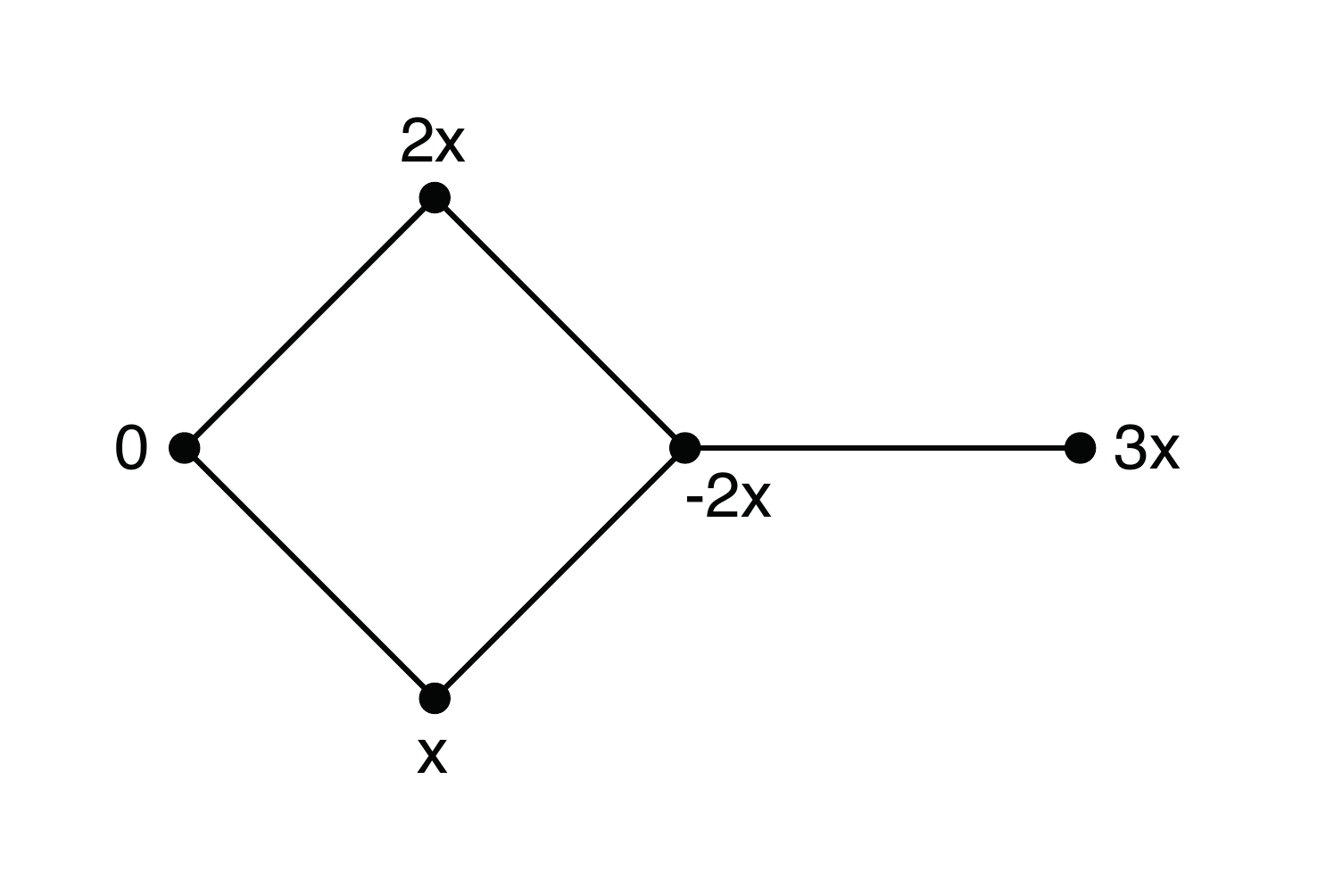}
        \caption{Graphs associated to $C=\{0,x,2x,-2x,3x\}$.}
    \label{fig:2tipo10}
\end{figure}
Hence the following holds:
\begin{proposition}
Suppose there exists a good $(mv,m,\Gamma,1)$-DF with $1\leq m\leq 5$ or $m=7$, where $\Gamma$ is one of the graphs of Figure \ref{fig:2tipo10}. Then there exists an optimal $(mv,5,3,1)$-OOC.
\end{proposition}

\section{Cyclotomic constructions for perfect $(p,5,3,1)$-OOCs with $p$ prime}

In this section, we present some constructions for optimal and perfect $(p,5,3,1)$-OOCs with $p$ prime. We begin by showing some theoretical arguments that guarantee the existence of such a code by using a classical tiling argument. 

The problem of additive tilings and factorisations of groups by subsets was first introduced by Haj\`os in 1941 \cite{Haj}, and since then it has been investigated in different frameworks, see \cite{SS}. The problem considered here can be viewed as a partial tiling problem in the cyclic group $\mathbb{Z}_{\ell}$.






\begin{TP} 
Given a pair $(\ell,S)$, where $\ell$ is a positive integer and $S$ is a $z$-subset of $\mathbb{Z}_\ell$, with $\ell = z\cdot t+r$, does there exist a $t$-set $T\subset \mathbb{Z}_\ell$ such that $|S+T| = z\cdot t$? 
\end{TP}
If  $r<z,t$,  $T$ is said to be \textit{optimal}, while  $T$ is said to be \textit{perfect} if $r=0$ .
For short, we will refer to the above problem as the  $(\ell,S)$-\emph{tiling problem}.

We point out that the above problem can also be equivalently stated in terms of packings, see the problem on page 459 of  \cite{Bu96}.

The idea of our construction is the following: starting from the finite field ${\rm GF}(p)$ with $p$ prime, we take $g$ to be a primitive element of ${\rm GF}(p)$ and define $D_j^\ell = g^j\langle{g^{\ell}}\rangle$ to be the $j^{th}$ cyclotomic class of order $\ell$, where $0 \leq j \leq \ell-1$. In other words, $D_0^{\ell}$ is the multiplicative subgroup containing $\ell^{th}$ powers in ${\rm GF}(p)$ and each $D_j^{\ell}$ is a unique coset of this subgroup. From the primitive element $g$ of ${\rm GF}(p)$, we compute the discrete logarithms of the set of differences of a codeword-set $C$ modulo $\ell$, and we determine the set of indexes $S$ corresponding to the cyclotomic classes where these differences belong. From $S$ we have to find a solution $T$ to the  $(\ell,S)$-tiling problem. 
If such a solution exists, a  $(p,5,3,1)$-OOC can then be constructed by considering the codeword-sets:
\[
 g^{z+t}\cdot C, \qquad \text{ for $z\in S,$ $t \in T$.}
\]
Under suitable conditions on the parameter $\ell$, see Remark \ref{rem:cond_param_OOC}, we can show that the constructed $(p,5,3,1)$-OOC will be optimal.  Note that in Example \ref{ex:explicit_constr} we show some explicit constructions of optimal OOCs with this method.

In the following lemma we characterise when two non-zero elements belong to distinct cyclotomic classes of order $\ell$.
\begin{lemma}\label{lem:nec_cond}
Let $p\geq 11$ be an odd prime and let $n\ell \mid p-1$ for some positive integers $n,\ell$.  Then for any distinct $x,y \in GF(p)$ \[  xD_0^\ell \cap  yD_0^\ell = \emptyset \quad \Leftrightarrow \quad \{xy^{-1}\} \not\in D_0^\ell.\]
\end{lemma}

Given  ${\rm GF}(p)$, an integer $\ell$ dividing $p-1$,  and a primitive element $g$ we denote by $Log_{p,\ell}$ the logarithmic map modulo $\ell$:
\[
Log_{p,\ell}: \quad  g^\alpha  \in  \mathbb{Z}_p^\times  \mapsto \alpha \in \mathbb{Z}_\ell.
\] 
Note that $Log_{p,\ell}(D_0^\ell) = 0$. 

\begin{proposition}\label{prop:discr_log}
Let $p\geq 11$ be an odd prime and let $n \mid \frac{p-1}{2}$. Let $\ell=\frac{p-1}{2n}$ and let $g$ be a primitive element of $\mathbb{Z}_p$. Let $C$ be a codeword-set with weight $w$, autocorrelation $\lambda_a$, and list of differences $\Delta C = \pm S$. 


 If $ \{xy^{-1}\} \not\in D_0^\ell$ for every distinct $x,y \in S$ and there exists a solution $T$ to the $(\ell,\ Log_{p,\ell}(S))$-tiling problem, then there exists a $(p,w,\lambda_a,1)$-OOC with $n|T|$ codewords.
\end{proposition}
\begin{proof}
Let $p,n$ and $\ell$ be as in the statement.
Since for every distinct $x,y \in S$, $ \{xy^{-1}\} \not\in D_0^\ell$  then,  from  Lemma \ref{lem:nec_cond},  $xD_0^\ell \cap  yD_0^\ell= \emptyset$ for every $x,y \in S$. Moreover, as $D_0^\ell$ has even order, $-1 \in D_0^\ell$,  hence we may write
 \[D_0^\ell = \pm \{x_1, x_2,\dotsc, x_n\}.
\]
Let now $g$ be a primitive element of ${\rm GF}(p)$, and let $T$ be a solution to the  $(\ell,\ Log_{p,\ell}(S))$-tiling problem; denote as $\{y_1,\dotsc, y_t\}$ a set of representatives for $(Log_{p,\ell})^{-1}(T)$. Construct the code $\mathcal{C}$ having as codeword-sets:
 \[
 (x_i  y_j) \cdot C, \qquad \text{for $i=1,\dotsc,n$ and $j=1,\dotsc, t$. }
 \]
Clearly, each of these codeword-sets has weight $w$ and autocorrelation $\lambda_a$, having as set of differences $\pm (x_i  y_j \cdot S)$. To show that the cross-correlation between two distinct codeword-sets is $1$, assume that two codeword-sets $C^1 = (x_i  y_j)\cdot C $ and   $C^2 = (x_u  y_v) \cdot C  $ of $\mathcal{C}$ share a common difference $a x_i  y_j = \pm bx_u  y_v $ for some $a,b\in S$. This implies that:
 \[
 Log_{p,\ell}(a x_i  y_j) =Log_{p,\ell}(a) + Log_{p,\ell}(y_j) =  Log_{p,\ell}(b) + Log_{p,\ell}(y_v) =  Log_{p,\ell}( \pm bx_u  y_v).
 \]
However, since $T$ is a solution to the   $(\ell,\ Log_{p,\ell}(S))$-tiling problem, we necessarily have $a=b$ and $y_j=y_v$, hence $ (x_i)^{-1}  x_u = \pm 1$ and $C^1=C^2$. 
\end{proof}
\begin{remark}\label{rem:coprimes}
Given any $a$ coprime with $\ell$, the existence of a solution to the $(\ell,S)$-tiling problem is equivalent to the existence of a solution to the $(\ell,a\cdot S)$-tiling problem. Hence previous construction does not depend on the choice of the primitive element.
\end{remark}

Let now $w=5$, $\lambda_a = 3$ and $|S|=5$. By writing $\ell=5t+r$ with
$0\leq r<5$, from Proposition \ref{prop:discr_log} we deduce that an optimal (resp. perfect) solution $T$ of the $(\ell,\ Log_{p,\ell}(S))$-tiling problem
yields an optimal (resp. perfect) $(p,5,3,1)$-OOC
whenever $nr<5$.
\begin{remark} \label{rem:cond_param_OOC}
Let $\ell = \frac{p-1}{2n}=5t+r$ for some $n\mid \frac{p-1}{2}$ and $0\leq r <5$. By examining the construction of Proposition \ref{prop:discr_log} in view of the upper bound of Theorem \ref{thm:bound}, we find the following conditions on the parameters for an optimal $(p,5,3,1)$-OOC:
\begin{center}
\begin{tabular}{c|l} 
$p \pmod{10}$ & $(n,r)$ \\ \hline
$1$ & $r=0$ and $n$ such that $10n\mid {p-1}$\\
$3$ & $(1,1)$ \\
$7$& $(1,3)$ or $(3,1)$ \\
$9$& $(1,4)$, $(2,2)$ or $(4,1)$ \\
\end{tabular}
\end{center}

\end{remark}

Note that if for a codeword-set $C$ its set of differences $\Delta C$ contains many elements in a geometric progression with common ratio $r$, the condition of Lemma \ref{lem:nec_cond} becomes quite easier to be verified. This is the case for the codeword-sets of  the following remark.
\begin{remark}\label{rem: multiset A}
In the following statements we will consider OOCs whose codeword-sets have two basic structures. In our constructions, for an odd prime $p \geq 11$ , a code $\mathcal{C}$ is formed by type $10$ codeword-sets  of the form:
\[
uC^1 = \{0,u,-u,3u,-3u\}\quad \text{ or}\quad uC^2=\{0,u,2u,-2u,3u\} \qquad u \in U \subset {\rm GF}(p);
\]
in the first case the list of differences  is  $\Delta(uC^1)= \pm\{u,2u,3u,4u,6u\}$, while in the second one  $\Delta(uC^2)= \pm\{u,2u,3u,4u,5u\}$. 

 If $p \equiv r \pmod{10}$, where $0\leq r < 10$, a  code $\mathcal{C}$ comprising $\lfloor\frac{p}{10}\rfloor$ codeword-sets of either type, say $uC^i$ for $u\in U \subset GF(p)$, is a $(p,5,3,1)$-OOC if $\bigcup_{u\in U} \Delta(uC^i)$ is a set containing $p-r$ distinct elements. Note that if $p \equiv 1 \pmod{10}$ this implies that $\bigcup_{u\in U} \Delta(uC^i)={\rm GF}(p)^*$.
\end{remark}

\begin{example}\label{ex:explicit_constr}
We construct an optimal  $(p,5,3,1)$-OOC for $p\in \{67,89,157\}$ based on the result of Proposition \ref{prop:discr_log}. In each case we have that the  codeword-set is of the form $uC$, with $C=\{0,1,-1,3,-3\}$, whose set of differences is $\pm \{1,2,3,4,6\}$.

We show in detail the construction for $p = 157$.  We choose $n=3$, hence $\ell=26$ and $D_0^{26} = \pm\{1,12,13\}$. Let $g=5$ be a primitive element of ${\rm GF}(157)^*$; we have $S = Log_{p,\ell}(\{1,2,3,4,6\}) = \{0,4,11,15,22\}$. It is an easy check to verify that $T=\{0,3,6,9,12\}$ is a solution to the $(26,S)$-tiling problem, as:
\[
S+T = \mathbb{Z}_{26}\setminus \{19\}.
\]
An optimal $(157,5,3,1)$-OOC is
\[
\{(\alpha \cdot \beta) C \mid \alpha \in \{1,12,13\}, \beta \in \{5^0,5^3,5^6,5^9,5^{12}\} \}.
\]

The interested reader may verify the cases $p \in \{67,89\}$ by following the table below; for completeness, we also reported the case $p=157$ shown above: 
\begin{center}
\begin{tabular}{ c|l|l|l}
$(p,n,\ell,g)$&$D_0^\ell$&$S$ &  $T$ \\ \hline
$(67,3,11,2)$&$\pm\{1,29,30\}$  &$\{0,1,2,6,7\}$& $\{0,3\}$ \\
$(89,2,22,7)$&$\pm\{1,34\}$&$\{0,3,4,7,8\}$& $\{0,2,11,13\}$ \\
$(157,3,26,5)$ & $\pm\{1,12,13\}$ & $\{0,4,11,15,22\}$ & $\{0,3,6,9,12\}$ \\
\end{tabular}
\end{center}
%
%

\end{example}


\begin{proposition}\label{perfect_nec_cond}
    Let $\mathcal{C}$ be a perfect $(p,5,3,1)$-OOC with $p = 10s+1$ prime and let $m$ denote the number of codeword-sets with differences $\pm \{u,2u,3u,4u,5u\}$. Then for every primitive element $g$:
    \[
    m Log_{p,5} \left(5 \cdot 6^{-1}\right) + s Log_{p,5}(144)\equiv 0 \pmod{5}.
    \]
\end{proposition}

\begin{proof}
    Let $\mathcal{C}$ be as in the statement; since the code is perfect, $\bar \Delta \mathcal{C}$ is $\mathbb{Z}_p^*$.  Consider the map $Log_{p,5}:   g^\alpha  \in  \mathbb{Z}_p^\times  \mapsto \alpha \in \mathbb{Z}_5$. From now on, we implicitly assume that the computations are performed in $\mathbb{Z}_5$.
It is easy to see that:
\[
\sum_{x \in \mathbb{Z}_p^*} Log_{p,5}(x) = \frac{p-1}{5}(0+1+2+3+4)  =0. \]
 Note that $Log_{p,5}(-1)=0$, hence $Log_{p,5}(-x)= Log_{p,5}(x) $. For a codeword-set $C$ with $\bar \Delta C=\pm \{u,2u,3u,4u,5u\}$, the sum of $Log_{p,5}$ of the elements in $\bar \Delta C$ is:
\[
\begin{aligned}
&2(Log_{p,5}(u) + Log_{p,5}(2u) +Log_{p,5}(3u) + Log_{p,5}(4u) +Log_{p,5}(5u)) \\ &= 10  Log_{p,5}(u) +2\left(3Log_{p,5}(2) + Log_{p,5}(3) + Log_{p,5}(5)  \right) \\
&=  2\left(3Log_{p,5}(2) + Log_{p,5}(3) + Log_{p,5}(5)  \right). 
\end{aligned}
\]
Similarly, for a codeword-set $C$ with $\bar \Delta C = \pm \{u,2u,3u,4u,6u\}$ the sum of $Log_{p,5}$ of the elements in $\bar \Delta C$ is:
\[
 4\left(2Log_{p,5}(2) + Log_{p,5}(3)\right). 
\]
Since $\mathcal{C}$ is a perfect code, $\sum_{C\in \mathcal{C}} \sum_{x \in \bar \Delta C}Log_{p,5}(x) =  \sum_{x \in \mathbb{Z}_p^*} Log_{p,5}(x) =0$, and $\mathcal{C}$ has $m$ and $s-m$ codeword-sets whose difference support are respectively  $\pm \{u,2u,3u,4u,5u\}$ and  $\pm \{u,2u,3u,4u,6u\}$:
\[
\begin{aligned}
\sum_{C\in \mathcal{C}} \sum_{x \in \bar \Delta C}Log_{p,5}(x) = 2m\left(3Log_{p,5}(2) + Log_{p,5}(3) + Log_{p,5}(5)  \right)  \\+ 2(s-m)\left(4Log_{p,5}(2) + 2 Log_{p,5}(3)\right)  = 0. 
\end{aligned}
\]
We then have 
\[
m(Log_{p,5}(5) - Log_{p,5}(2) - Log_{p,5}(3)) + s (4Log_{p,5}(2)+2Log_{p,5}(3) ) =0,
\]
which is equivalent to the statement. 
\end{proof}
Note that also here it does not depend on the chosen primitive element.

    Given a prime $p\equiv 1 \pmod{10}$, from Proposition \ref{perfect_nec_cond} we can directly determine modulo $5$ the number $m$ of codeword-sets with differences $\pm \{u,2u,3u,4u,5u\}$ from the values of $Log_{p,5}(5\cdot 6^{-1}) $ and $ Log_{p,5}(144)$. Moreover, the following corollary immediately follows:
    \begin{corollary}
        Let $p \geq 31$ be a prime with $5\cdot 6^{-1} \in D_0^5$. If $p\not \equiv 1 \pmod{50}$ and $144 \pmod{p}$ is not a fifth power, then there does not exist a perfect $(p,5,3,1)$-OOC.
    \end{corollary}
From this corollary, we have that there does not exist a perfect $(p,5,3,1)$-OOC for
\[
p \in \{31, \ 61, \   191, \   311,\    491,\    821,\    881,\ 1021, \   1381,\   1811,\dotsc\}
\]

In what follows, we present two cyclotomic constructions of perfect $(v,5,3,1)$-OOCs, showing the  cyclotomic conditions that must be met to obtain examples of perfect $(v,5,3,1)$-OOCs using these constructions. \\

We begin by considering perfect $(p,5,3,1)$-OOCs consisting of codeword-sets of the form $C = \{0,u,-u,3u,-3u\}$, where $u \in {\rm GF}(p)$.

\begin{lemma}\label{lemma: generalised coset reps cycle1}
    Let $p \equiv 1 \pmod{10}$ be a prime. Let $5^r$ be the largest $5^{th}$ power dividing $p-1$, and let $i \in \{1,2,\ldots,r\}$. Then $\{1,2,3,4,6\}$ is a complete system of coset representatives for $D_0^{5^i}$ 
    in $D_0^{5^{i-1}}$ 
    if and only if $2 \in D_{(5^{i-1})j}^{5^i}$ for $j \neq 0$ and $12 \in D_0^{5^i}$. 
\end{lemma}

\begin{proof}
     In ${\rm GF}(p)^*$, consider the cyclotomic classes $D_0^{5^i}$ and $D_0^{5^{i-1}}$; note that $D_0^{5^i}$ is a subgroup $D_0^{5^{i-1}}$ of index $5$. Let $x \in D_{5^{i-1}i_x}^{5^i}$, where $x \in \{1,2,3,4,6\}$.   Clearly when $x = 1$, $D_{5^{i-1}i_1}^{5^i} = D_0^{5^i}$. \\ 
    For the forward implication, assume that $\{1,2,3,4,6\}$ is a complete system of coset representatives of $D_0^{5^i}$ in $D_0^{5^{i-1}}$. Then $2 \in D_{5^{i-1}i_2}^{5^{i}}$ and $3 \in D_{5^{i-1}i_3}^{5^{i}}$, where $1 \leq i_2 \neq i_3 \leq 4$. It then follows that $4 \in D_{5^{i-1}i_4}^{5^{i}} = D_{2(5^{i-1})i_2}^{5^i}$ and $6 \in D_{5^{i-1}i_6}^{5^i}$, where $i_6 = i_2 + i_3$. It can be checked that the only possibilities for $(i_1,i_2,i_3,i_4,i_6)$ are $(0,1,3,2,4)$, $(0,2,1,4,3)$, $(0,3,4,1,2)$ and $(0,4,2,3,1)$. In all cases, $2 \in D_{(5^{i-1})j}^{5^i}$ for $j \neq 0$ and $i_2 + i_6 = i_3 + i_4 = 0$, which implies $12 \in D_0^{5^i}$.  \\
    For the reverse direction, $2 \in D_{5^{i-1}i_2}^{5^{i}}$, where $1 \leq i_2 \leq 4$, implies that $i_4 = 2i_2$. As $12 \in D_{5^{i-1}(i_2+i_6)}^{5^i} = D_{5^{i-1}(i_3+i_4)}^{5^i} = D_0^{5^i}$, we can write $i_2 + i_6 = i_3 + i_4 \equiv 0 \pmod 5$. From $i_2 + i_6 \equiv 0 \pmod 5$, we obtain $i_6 = -i_2 = 4i_2$. Moreover, using the identity $i_4 = 2i_2$, we may rewrite the equation $i_2 + i_6 = i_3 + i_4$ as $i_6 = i_3 + i_2$, thus obtaining $i_3 = 3i_2$ using the identity $i_6 = 4i_2$. We can then see that $(i_1,i_2,i_3,i_4,i_5)$ is a complete set of coset representations for $D_0^{5^i}$ in $D_0^{5^{i-1}}$, completely determined by the value of $i_2$. In fact, the  possible options for $(i_1,i_2,i_3,i_4,i_5)$ are $(0,1,3,2,4)$, $(0,2,1,4,3)$, $(0,3,4,1,2)$, $(0,4,2,3,1)$.    
\end{proof}

By combining Remark \ref{rem: multiset A} and Lemma \ref{lemma: generalised coset reps cycle1}, we get the following family of perfect $(p,5,3,1)$-OOCs. 

\begin{theorem}\label{theorem: generalised cycle conditions}
    Let $p \equiv 1 \pmod{10}$ be a prime and let $5^r$ be the largest power of $5$ which divides $p-1$. If $2 \in D_{(5^{i-1})j}^{5^i}$ for $j \neq 0$ and $12 \in D_0^{5^i}$ for some $i$ in $\{1,2,\ldots,r\}$, then there exists a perfect $(p,5,3,1)$-OOC.
\end{theorem}

\begin{proof}
    Let $i \in \{1,2,\ldots,r\}$ and suppose that $\{1,2,3,4,6\}$ is a complete system of representatives for $D_0^{5^i}$ in $D_0^{5^{i-1}}$. Then it follows that $\{1,2,3,4,6\}*D_0^{5^i}  = D_0^{5^{i-1}}$. As $5^i$ is odd, it follows that $\{1,-1\}$ is a subgroup of $D_0^{5^i}$. Let $S$ be a complete system of representatives for the cosets of $\{1,-1\} \in D_0^{5^i}$, that is $\{1,-1\}*S= D_0^{5^i}$. Finally, let $T$ be a complete system of representatives for the cosets of $D_0^{5^{i-1}} \in {\rm GF}(p)^*$, that is $D_0^{5^{i-1}}*T = {\rm GF}(p)^*$. Set $C=\{0,3,-3,1,-1\}$ and $U = S*T$. Then $\{C*u \,|\, u \in U\}$ is a perfect $(p,5,3,1)$-OOC.
\end{proof}

The following example is the smallest non-trivial example of the construction outlined in Theorem \ref{theorem: generalised cycle conditions}.

\begin{example}
    In ${\rm GF}(211)$, when $2$ is used to multiplicatively generate ${\rm GF}(211)^*$, there is a perfect $(211,5,3,1)$-OOC consisting of $\lfloor\frac{211}{10}\rfloor = 21$ blocks of the form $\{0,u,-u,3u,-3u\}$, where $u \in S = \{1,2^{5},\ldots,2^{95}\}$. By Theorem \ref{theorem: generalised cycle conditions}, this follows from the fact that when $2$ is the multiplicative generator of this field then $1 \in D_0^5$, $2 \in D_1^5$, $4 \in D_2^5$, $3 \in D_3^5$ and $6 \in D_4^5$.
\end{example}

In the special case where $i = 1$, using Theorem \ref{theorem: generalised cycle conditions}, we can generate the following $(p,5,3,1)$-OOC in finite fields of order less than $1000$. 

\begin{center}
    \begin{tabular}{c|c}
        $p$ & Codeword-sets \\
        \hline 
        $11$ & $\{0,1,3,8,10\}$ \\ 
        $211$ & $\{0,2^{5d},3(2^{5d}),208(2^{5d}),210(2^{5d})\}\ \textrm{for}\  0 \leq d \leq 20$ \\
        $271$ & $\{0,6^{5d},3(6^{5d}),268(6^{5d}),270(6^{5d})\}\ \textrm{for}\ 0 \leq d \leq 26$ \\
        $811$ & $\{0,3^{5d},3(3^{5d}),808(3^{5d}),810(3^{5d})\}\ \textrm{for}\ 0 \leq d \leq 80$ \\
        $941$ & $\{0,2^{5d},3(2^{5d}),938(2^{5d}),940(2^{5d})\}\  \textrm{for}\ 0 \leq d \leq 93$ \\
    \end{tabular}
\end{center}

For $2\leq i \leq 5$, the smallest $p$ for which a $(p,5,3,1)$-OOC  can be constructed in each case using Theorem \ref{theorem: generalised cycle conditions} is listed in the table below:

\begin{center}
    \begin{tabular}{c|c}
        $p$  & $i$ \\
        \hline 
        $77101$ & $2$ \\ 
        $1832251$ & $3$ \\
        $1948481251$ & $4$ \\
        $159246187501$ & $5$
    \end{tabular}
\end{center}

We now repeat the previous process to identify the cyclotomic conditions that need to be met in order to construct  perfect $(p,5,3,1)$-OOCs consisting of codeword-sets of the form $C = \{0,u,2u,-2u,3u\}$, where $u \in {\rm GF}(p)$.  
\begin{lemma}\label{lemma: generalised coset reps cycle}
    Let $p \equiv 1 \pmod{10}$ be a prime. Let $5^r$ be the largest $5^{th}$ power dividing $p-1$, and let $i \in \{1,2,\ldots,r\}$. Then $\{1,2,3,4,5\}$ is a complete system of coset representatives for $D_0^{5^i}$ 
    in $D_0^{5^{i-1}}$ 
    if and only if $2 \in D_{(5^{i-1})j}^{5^i}$ for $j \neq 0$ and either: 
    \begin{itemize}
        \item [(a)] $6,20 \in D_0^{5^i}$ or
        \item [(b)] $10,12 \in D_0^{5^i}$.
    \end{itemize} 
\end{lemma}

\begin{proof}
     In ${\rm GF}(p)^*$, consider the cyclotomic classes $D_0^{5^i}$ and $D_0^{5^{i-1}}$; note that $D_0^{5^i}$ is a subgroup $D_0^{5^{i-1}}$ of index $5$. Let $x \in D_{5^{i-1}i_x}^{5^i}$, where $x \in \{1,2,3,4,5\}$.   Clearly when $x = 1$, $D_{5^{i-1}i_1}^{5^i} = D_0^{5^i}$. \\ 
    The proof of the forward implication is analogous to the proof of the forward implication in Lemma \ref{lemma: generalised coset reps cycle1}. It can be checked that the only possibilities for $(i_1,i_2,i_3,i_4,i_5)$ are: $(0,1,3,2,4)$, $(0,1,4,2,3)$, $(0,2,1,4,3)$, $(0,2,3,4,1)$, $(0,3,2,1,4)$, $(0,3,4,1,2)$, $(0,4,1,3,2)$ and $(0,4,2,3,1)$. In all cases, $2 \in D_0^{5^{i-1}}\backslash{D_0^{5^i}}$ and, further, we can deduce that either $6,20 \in D_0^{5^i}$ or $10,12 \in D_0^{5^i}$. \\
    For the reverse direction, we prove case $(a)$; case $(b)$ follows analogously. As $2 \in D_{5^{i-1}i_2}^{5^{i}}$, where $1 \leq i_2 \leq 4$, this implies that $i_4 = 2i_2 \neq 0$. As $6 \in D_{5^{i-1}(i_2+i_3)}^{5^i} = D_0^{5^i}$, we can write $i_2 + i_3 \equiv 0 \pmod 5$, and therefore $i_3 = -i_2 = 4i_2$. Moreover, as $20 \in D_{5^{i-1}(i_4+i_5)}^{5^i} = D_0^{5^i}$, it follows that $i_5 = -2i_2 = 3i_2$. We can therefore see that $(i_1,i_2,i_3,i_4,i_5)$ is a complete set of coset representations for $D_0^{5^i}$ in $D_0^{5^{i-1}}$, completely determined by the value of $i_2$. In fact, the possible options for $(i_1,i_2,i_3,i_4,i_5)$ are $(0,1,3,2,4)$, $(0,2,1,4,3)$, $(0,3,4,1,2)$, $(0,4,2,3,1)$.    
\end{proof}

We can then use cyclotomic conditions to establish the following family of perfect $(p,5,3,1)$-OOCs. The proof of this result is analogous to the proof of Theorem \ref{theorem: generalised cycle conditions}.

\begin{theorem}\label{theorem: generalised bull conditions}
    Let $p \equiv 1 \pmod{10}$ be a prime and let $5^r$ be the largest power of $5$ which divides $p-1$. If $2 \in D_0^{5^{i-1}}\backslash D_0^{5^i}$ and either $6,20 \in D_0^{5^i}$ or $10,12 \in D_0^{5i}$ for some $i$ in $\{1,2,\ldots,r\}$, then there exists a perfect $(p,5,3,1)$-OOC.
\end{theorem}

In finite fields of order less than $1000$, we can obtain the following $(p,5,3,1)$-OOCs using Theorem \ref{theorem: generalised bull conditions}. Note that $i=1$ in all of these cases.  

\begin{center}
    \begin{tabular}{c|c}
        $p$ & Codeword-sets \\
        \hline 
        $11$ & $\{0,1,2,3,9\}$ \\ 
        $421$ & $\{0,2^{5d},2(2^{5d}),3(2^{5d}),419(2^{5d})\}\ \textrm{for}\ 0 \leq d \leq 42$ \\
        $701$ & $\{0,2^{5d},2(2^{5d}),3(2^{5d}),699(2^{5d})\}\ \textrm{for}\ 0 \leq d \leq 70$ \\
    \end{tabular}
\end{center}

For $i =2,3$, the smallest $p$ for which a $(p,5,3,1)$-OOC can be generated in each case is given in the table below.

\begin{center}
    \begin{tabular}{c|c}
       $p$ & $i$ \\
        \hline 
        $3878701$ & $2$ \\ 
        $1550857001$ & $3$ \\
    \end{tabular}
\end{center}

\section{A recursive construction}\label{sec:recursive}
In this section we propose a general recursive construction for optimal or perfect $(v,5,3,1)$-OOCs. 

\begin{proposition}\label{prop:recursive}
    If there exists a perfect $(v,5,3,1)$-OOC with all the blocks 
    of the form $\{0,x,-x,3x,-3x\}$ and/or $\{0,x,2x,-2x,3x\}$ and/or $\{0,x,2x,4x,5x\}$ 
    and a perfect (respectively, optimal) $(w,5,3,1)$-OOC with $\gcd(w,30)=1$, then there exists a perfect (respectively, optimal) $(vw,5,3,1)$-OOC.
    \end{proposition}
    \begin{proof}
       Let $\mathcal{C}_1$ be a perfect $(v,5,3,1)$-OOC as in the statement and let $\mathcal{C}_2$ be a perfect (respectively, optimal) $(w,5,3,1)$-OOC with $\gcd(w,30)=1$. 
      
       For any $C\in \mathcal{C}_1$ and for any $y\in \mathbb{Z}_w$ set
$$C_y=\begin{cases}
    \{0,x+vy,-x-vy,3x+3vy,-3x-3vy\}\,&{\rm if }\,C=\{0,x,-x,3x,-3x\}, \\
    \{0,x+vy,2x+2vy,-2x-2vy,3x+3vy\}\, &{\rm if }\, C=\{0,x,2x,-2x,3x\},\\
    \{0,x+vy,2x+2vy,4x+4vy,5x+5vy\}\, &{\rm if } \,C=\{0,x,2x,4x,5x\}.
\end{cases}$$
Note that the support of $C_y$ is:
$$\bar\Delta C_y=\pm\{x+vy,2x+2vy,3x+3vy,4x+4vy, 6x+6vy\}$$ if $C=\{0,x,-x,3x,-3x\}$, while we have
$$\bar\Delta C_y=\pm\{x+vy,2x+2vy,3x+3vy,4x+4vy, 5x+5vy\}$$
if $C=\{0,x,2x,-2x,3x\}$ or $C=\{0,x,2x,4x,5x\}$.
In all cases, the reader can verify that no element occurs more than three times in the multiset $\Delta C_y$.

Now set:
$$\mathcal{C}_3=\{C_y \mid C\in \mathcal{C}_1, y\in \mathbb{Z}_w\} \quad
\textrm{and}\quad
\mathcal{C}_4=\{vB \mid B\in \mathcal{C}_2\}.$$
It can be verified that $\mathcal{C}=\mathcal{C}_3 \cup \mathcal{C}_4$ is a perfect (respectively, optimal) $(vw,5,3,1)$-OOC. We have already observed that every element in the support of $C_y$ occurs with multiplicity at most $3$ in $\Delta C_y$, for every $C_y \in \mathcal{C}_3$. Further, since $\mathcal{C}_4$ is obtained by embedding $\mathcal{C}_2$ into $\mathbb{Z}_{vw}$, it is immediate that for any $B \in \mathcal{C}_2$, any element in the support of $vB$ occurs at most $3$ times in $\Delta(vB)$.

Now we prove that for any two sets $X,Y \in \mathcal{C}$, $\Delta X \cap \Delta Y = \emptyset$. Obviously, ${\rm H} = \langle{v}\rangle = \{0,v,2v,\ldots,(w-1)v\}$ is a subgroup of $\mathbb{Z}_{vw}$. Note that each element of $vB \subset {\rm H}$, and moreover, observe that the supports of the blocks of $\mathcal{C}_4 = \{vB:B \in \mathcal{C}_2\}$ are disjoint.

For each $C_y \in \mathcal{C}_3$, we next show that $\Delta C_y \cap {\rm H} = \emptyset$. Every difference in $\Delta C_y$ takes the form $i(x+vy)$, where $i$ is determined by the block type. It is immediate that $ivy \in {\rm H}$, therefore $i(x+vy)$ is in a non-trivial coset of H if and only if $ix \not\equiv 0 \pmod v$ for every $i$. Since $v \nmid ix$ (this comes from the fact that $ix \in \mathbb{Z}_v^*$), it is immediate that the supports of $\mathcal{C}_3$ are disjoint from the blocks of $\mathcal{C}_4$.

It remains to prove that $\Delta C_y \cap \Delta C_{y'} = \emptyset$ for distinct blocks $C_y, C_{y'} \in \mathcal{C}_3$. This can be achieved by establishing that $$i(x+vy) \not\equiv \pm j(x'+vy') \pmod {vw}$$ for all $x,x' \in \mathbb{Z}_v$, $y,y' \in \mathbb{Z}_w$ and $i,j$ determined by the block type. In the case where $ix = jx'$, the differences are only the same if $vy = vy'$, since the elements of the subgroup ${\rm H} = \{vy \mid y \in \mathbb{Z}_w\}$, this happens only if $y = y' \in \mathbb{Z}_w$. When $ix \neq jx'$, then $$i(x+vy) \equiv \pm j(x'+vy') \pmod {vw}$$ if and only $ix$ and $jx'$ lie in the same coset of ${\rm H}$. Since each $C \in \mathcal{C}_1$ guarantees that $ix \neq \pm jx' \pmod v$, it follows that the elements $ix$ and $jx'$ must lie in distinct cosets of ${\rm H}$.
\end{proof}

As an immediate consequence of previous proposition we have the following:
    \begin{corollary}
    If there exist a perfect $(v,5,3,1)$-OOC with all the blocks of the form $\{0,x,-x,3x,-3x\}$ and a perfect (respectively, optimal) $(w,5,3,1)$-OOC with $\gcd(w,6)=1$, then there exists a perfect (respectively, optimal) $(vw,5,3,1)$-OOC.
\end{corollary}

    \begin{example}
Since $\mathcal{C}_1=\{\{0,1,-1,3,-3\}\}$ is a perfect $(11,5,3,1)$-OOC, if we take  $\mathcal{C}_1=\mathcal{C}_2$,
following the proof of the previous proposition, we get 
 $\mathcal{C}_3=\{\{0, 12, 109, 36, 85\},$ $\{0, 23, 98, 69, 52\},$  $ \{0, 34, 87, 102, 19\}, \{0, 45, 76, 14, 107\},$ $ \{0, 56, 65, 47, 74\}, \{0, 67, 54, 80, 41\},$ $\{0, 78, 43, 113, 8\},$ $\{0, 89, 32, 25, 96\}, \{0, 100, 21, 58, 63\},$ $\{0, 111, 10, 91, 30\}, 
\{0, 1, 120, 3, 118\}\}$ and $\mathcal{C}_4=\{\{0, 11, 110, 33, 88\}\}.$
Then $\mathcal{C}_3 \cup \mathcal{C}_4$ is a perfect $(121,5,3,1)$-OOC.

Note that $\{\{0,1,-1,3,-3\}\}$ is also an optimal $(13,5,3,1)$-OOC, hence $\mathcal{C}_3 \cup \mathcal{C}_4$ is also an optimal $(143,5,3,1)$-OOC.
\end{example}

\begin{example}
As already remarked $\mathcal{C}_1=\{\{0,1,-1,3,-3\}\}$ is a perfect $(11,5,3,1)$-OOC. Additionally, we take $$\mathcal{C}_2= \{
\{0, 1, 2, 3, 26 \},
\{0, 9, 19, 29, 38 \},
\{0, 5, 12, 17, 33 \},
\{0, 4, 8, 22, 35 \}
 \}$$ an optimal $(49,5,3,1)$-OOC obtained from \cite{BT16}.\\ Following the proof of Proposition \ref{prop:recursive}, we get 

\begin{center}
 \begin{small}
\begin{tabular}{|llll|}\hline
\multicolumn{4}{|c|}{$\mathcal{C}_3$} \\ \hline
$\{ 0, 1, 3, 536, 538 \},$&
$\{ 0, 12, 36, 503, 527 \},$&
$\{ 0, 23, 69, 470, 516 \},$&
$\{ 0, 34, 102, 437, 505 \}$ \\
$\{ 0, 45, 135, 404, 494 \}$,&
$\{ 0, 56, 168, 371, 483 \}$,&
$\{ 0, 67, 201, 338, 472 \}$,&
$\{ 0, 78, 234, 305, 461 \}$, \\
$\{ 0, 89, 267, 272, 450 \}$,&
$\{ 0, 100, 239, 300, 439 \}$,&
$\{ 0, 111, 206, 333, 428 \}$,&
$\{ 0, 122, 173, 366, 417 \}$, \\
$\{ 0, 133, 140, 399, 406 \}$,&
$\{ 0, 107, 144, 395, 432 \}$,&
$\{ 0, 74, 155, 384, 465 \}$,&
$\{ 0, 41, 166, 373, 498 \}$, \\
$\{ 0, 8, 177, 362, 531 \}$,&
$\{ 0, 25, 188, 351, 514 \}$,&
$\{ 0, 58, 199, 340, 481 \}$,&
$\{ 0, 91, 210, 329, 448 \}$, \\
$\{ 0, 124, 221, 318, 415 \}$,&
$\{ 0, 157, 232, 307, 382 \}$,&
$\{ 0, 190, 243, 296, 349 \}$,&
$\{ 0, 223, 254, 285, 316 \}$, \\
$\{ 0, 256, 265, 274, 283 \}$,&
$\{ 0, 250, 263, 276, 289 \}$,&
$\{ 0, 217, 252, 287, 322 \}$,&
$\{ 0, 184, 241, 298, 355 \}$, \\
$\{ 0, 151, 230, 309, 388 \}$,&
$\{ 0, 118, 219, 320, 421 \}$,&
$\{ 0, 85, 208, 331, 454 \}$,&
$\{ 0, 52, 197, 342, 487 \}$, \\
$\{ 0, 19, 186, 353, 520 \}$,&
$\{ 0, 14, 175, 364, 525 \}$,&
$\{ 0, 47, 164, 375, 492 \}$,&
$\{ 0, 80, 153, 386, 459 \}$, \\
$\{ 0, 113, 142, 397, 426 \}$,&
$\{ 0, 131, 146, 393, 408 \}$,&
$\{ 0, 120, 179, 360, 419 \}$,&
$\{ 0, 109, 212, 327, 430 \}$, \\
$\{ 0, 98, 245, 294, 441 \}$,&
$\{ 0, 87, 261, 278, 452 \}$,&
$\{ 0, 76, 228, 311, 463 \}$,&
$\{ 0, 65, 195, 344, 474 \}$, \\
$\{ 0, 54, 162, 377, 485 \}$,&
$\{ 0, 43, 129, 410, 496 \}$,&
$\{ 0, 32, 96, 443, 507 \}$,& 
$\{ 0, 21, 63, 476, 518 \}$, \\
$\{ 0, 10, 30, 509, 529 \}$. &&& \\
\hline
\end{tabular}
\end{small}
\end{center}

and $$\mathcal{C}_4= \{
\{0, 11, 22, 33, 286 \},
\{0, 99, 209, 319, 418 \},
\{0, 55, 132, 187, 363 \},\\
\{0, 44, 88, 242, 385 \}
\}.$$
Then $\mathcal{C}_3 \cup \mathcal{C}_4$ is an optimal $(539,5,3,1)$-OOC.
\end{example}

\begin{example}
Take again $\mathcal{C}_1=\{\{0,1,-1,3,-3\}\}$ as a perfect $(11,5,3,1)$-OOC. Additionally, we take $$\mathcal{C}_2= \{
\{ 0, 1, 2, 3, 5 \},
\{ 0, 10, 20, 30, 50 \},
\{ 0, 12, 24, 36, 60 \},
\{ 0, 8, 31, 39, 62 \},$$
$$\{ 0, 6, 19, 32, 64 \},
\{ 0, 9, 18, 52, 61 \},
\{ 0, 7, 14, 21, 35 \}
 \}$$ an optimal $(77,5,3,1)$-OOC obtained from \cite{BT16}. Following the proof of Proposition \ref{prop:recursive}, we get

\begin{center}
\begin{small}
\begin{tabular}{|llll|}\hline
\multicolumn{4}{|c|}{$\mathcal{C}_3$} \\ \hline
$\{ 0, 1, 3, 844, 846 \}$,&
$\{ 0, 12, 36, 811, 835 \}$,&
$\{ 0, 23, 69, 778, 824 \}$,&
$\{ 0, 34, 102, 745, 813 \}$, \\ 
$\{ 0, 45, 135, 712, 802 \}$,&
$\{ 0, 56, 168, 679, 791 \}$,&
$\{ 0, 67, 201, 646, 780 \}$,&
$\{ 0, 78, 234, 613, 769 \}$, \\
$\{ 0, 89, 267, 580, 758 \}$,&
$\{ 0, 100, 300, 547, 747 \}$,&
$\{ 0, 111, 333, 514, 736 \}$,&
$\{ 0, 122, 366, 481, 725 \}$, \\
$\{ 0, 133, 399, 448, 714 \}$,&
$\{ 0, 144, 415, 432, 703 \}$,&
$\{ 0, 155, 382, 465, 692 \}$,&
$\{ 0, 166, 349, 498, 681 \}$, \\
$\{ 0, 177, 316, 531, 670 \}$,&
$\{ 0, 188, 283, 564, 659 \}$,&
$\{ 0, 199, 250, 597, 648 \}$,&
$\{ 0, 210, 217, 630, 637 \}$, \\
$\{ 0, 184, 221, 626, 663 \}$,&
$\{ 0, 151, 232, 615, 696 \}$,&
$\{ 0, 118, 243, 604, 729 \}$,&
$\{ 0, 85, 254, 593, 762 \}$, \\
$\{ 0, 52, 265, 582, 795 \}$,&
$\{ 0, 19, 276, 571, 828 \}$,&
$\{ 0, 14, 287, 560, 833 \}$,&
$\{ 0, 47, 298, 549, 800 \}$, \\
$\{ 0, 80, 309, 538, 767 \}$,&
$\{ 0, 113, 320, 527, 734 \}$,&
$\{ 0, 146, 331, 516, 701 \}$,&
$\{ 0, 179, 342, 505, 668 \}$, \\
$\{ 0, 212, 353, 494, 635 \}$,&
$\{ 0, 245, 364, 483, 602 \}$,&
$\{ 0, 278, 375, 472, 569 \}$,&
$\{ 0, 311, 386, 461, 536 \}$, \\
$\{ 0, 344, 397, 450, 503 \}$,&
$\{ 0, 377, 408, 439, 470 \}$,&
$\{ 0, 410, 419, 428, 437 \}$,&
$\{ 0, 404, 417, 430, 443 \}$, \\
$\{ 0, 371, 406, 441, 476 \}$,&
$\{ 0, 338, 395, 452, 509 \}$,&
$\{ 0, 305, 384, 463, 542 \}$,&
$\{ 0, 272, 373, 474, 575 \}$, \\
$\{ 0, 239, 362, 485, 608 \}$,&
$\{ 0, 206, 351, 496, 641 \}$,&
$\{ 0, 173, 340, 507, 674 \}$,&
$\{ 0, 140, 329, 518, 707 \}$, \\
$\{ 0, 107, 318, 529, 740 \}$,&
$\{ 0, 74, 307, 540, 773 \}$,&
$\{ 0, 41, 296, 551, 806 \}$,&
$\{ 0, 8, 285, 562, 839 \}$, \\
$\{ 0, 25, 274, 573, 822 \}$,&
$\{ 0, 58, 263, 584, 789 \}$,&
$\{ 0, 91, 252, 595, 756 \}$,&
$\{ 0, 124, 241, 606, 723 \}$, \\
$\{ 0, 157, 230, 617, 690 \}$,&
$\{ 0, 190, 219, 628, 657 \}$,&
$\{ 0, 208, 223, 624, 639 \}$,&
$\{ 0, 197, 256, 591, 650 \}$, \\
$\{ 0, 186, 289, 558, 661 \}$,&
$\{ 0, 175, 322, 525, 672 \}$,&
$\{ 0, 164, 355, 492, 683 \}$,&
$\{ 0, 153, 388, 459, 694 \}$, \\
$\{ 0, 142, 421, 426, 705 \}$,&
$\{ 0, 131, 393, 454, 716 \}$,&
$\{ 0, 120, 360, 487, 727 \}$,&
$\{ 0, 109, 327, 520, 738 \}$, \\
$\{ 0, 98, 294, 553, 749 \}$,&
$\{ 0, 87, 261, 586, 760 \}$,&
$\{ 0, 76, 228, 619, 771 \}$,&
$\{ 0, 65, 195, 652, 782 \}$, \\
$\{ 0, 54, 162, 685, 793 \}$,&
$\{ 0, 43, 129, 718, 804 \}$,&
$\{ 0, 32, 96, 751, 815 \}$,&
$\{ 0, 21, 63, 784, 826 \}$, \\
$\{ 0, 10, 30, 817, 837 \}$. &&& \\
\hline
\end{tabular}
\end{small}
\end{center}

and $$\mathcal{C}_4= \{
\{ 0, 11, 22, 33, 55 \},
\{ 0, 110, 220, 330, 550 \},
\{ 0, 132, 264, 396, 660 \},
\{ 0, 88, 341, 429, 682 \},$$
$$\{ 0, 66, 209, 352, 704 \},
\{ 0, 99, 198, 572, 671 \},
\{ 0, 77, 154, 231, 385 \}
\}.$$
Then $\mathcal{C}_3 \cup \mathcal{C}_4$ is an optimal $(847,5,3,1)$-OOC.
\end{example}
 \begin{remark}
    Note that a class of perfect $(v,5,3,1)$-OOCs with all the blocks of the form $\{0,x,-x,3x,-3x\}$ has been constructed in Theorem \ref{theorem: generalised cycle conditions}, while a class of perfect $(v,5,3,1)$-OOCs with all the blocks of the form $\{0,x,2x,-2x,3x\}$ has been constructed in Theorem \ref{theorem: generalised bull conditions}. This ensures that Proposition \ref{prop:recursive} gives an infinite family of new codes.
\end{remark} 

\subsection*{Acknowledgments}
This work arose from the workshop on ``Identifying constructions of Optical Orthogonal codes in which the weight of the sequence is not fixed" funded by the Heilbronn Focused Research Grant. \\
We would like to thank Sophie Huczynska and Siaw-Lynn Ng for their helpful discussions and contributions to the early results presented in this paper.

\end{document}